\documentclass[a4paper,reqno,12pt]{amsart} 
\usepackage{amsmath} 

\usepackage{amssymb}      

\usepackage{yhmath}
\usepackage{mathdots}
\usepackage{MnSymbol}

\usepackage{amsthm}      

\usepackage{tikz,amsmath}
\usetikzlibrary{arrows}

\usepackage{epsfig}      

\usepackage{graphicx}
\usepackage[normalem]{ulem}
\usepackage[all,line,
]{xy} 
\CompileMatrices 

\newcommand{\Ad}{\operatorname{Ad}}

\newcommand{\id}{\operatorname{id}}

\newcommand{\Span}{\operatorname{Span}}
\newcommand{\Tr}{\operatorname{Tr}}

\newcommand{\Geo}{\operatorname{Geo}}

\newcommand{\Br}{\operatorname{Br}}

\newcommand{\KMS}{\operatorname{KMS}}
\newcommand{\GS}{\operatorname{GS}}
\newcommand{\CS}{\operatorname{CS}}
\newcommand{\AF}{\operatorname{AF}}
\newcommand{\UHF}{\operatorname{UHF}}

   \theoremstyle{plain}
   \newtheorem{thm}{Theorem}[section]
   \newtheorem{prop}[thm]{Proposition}
   \newtheorem{lemma}[thm]{Lemma}  
   \newtheorem{cor}[thm]{Corollary}
   \theoremstyle{definition}

   \newtheorem{example}[thm]{Example}
   \theoremstyle{remark}

\usepackage{mdframed,xcolor}

\definecolor{mybgcolor}{gray}{0.8}
\definecolor{myframecolor}{rgb}{.647,.129,.149}

\mdfdefinestyle{mystyle}{
  usetwoside=false,
  skipabove=0.6em plus 0.8em minus 0.2em,
  skipbelow=0.6em plus 0.8em minus 0.2em,
  innerleftmargin=.25em,
  innerrightmargin=0.25em,
  innertopmargin=0.25em,
  innerbottommargin=0.25em,
  leftmargin=-.75em,
  rightmargin=-0em,
  topline=false,
  rightline=false,
  bottomline=false,
  leftline=false,
  backgroundcolor=mybgcolor,
  splittopskip=0.75em,
  splitbottomskip=0.25em,
  innerleftmargin=0.5em,
  leftline=true,
  linecolor=myframecolor,
  linewidth=0.25em,
}

\newmdenv[style=mystyle]{important}

\usepackage{lipsum}

   \numberwithin{equation}{section}

        \date{\today}

\title[Ground states on UHF algebras]{Ground states for generalized gauge actions on UHF algebras}
\author{Klaus Thomsen}

\date{\today}

\email{matkt@math.au.dk}
\address{Department of Mathematics, Aarhus University, Ny Munkegade, 8000 Aarhus C, Denmark}

\begin{document}

\maketitle

\begin{abstract} We describe the structure of ground states and ceiling states for generalized gauge actions on a UHF algebra. It is shown that both sets are affinely homeomorphic to the state space of a unital AF algebra, and that any pair of unital AF algebras can occur in this way, independently of the field of KMS states. In addition we study the subset of the ground states called $\KMS_\infty$-states.
 \end{abstract}

\section{Introduction and statement of the main result} 
 
Let $A$ be a unital $C^*$-algebra and $\alpha = \left(\alpha_t\right)_{t \in \mathbb R}$ a continuous one-parameter group of automorphisms on $A$; in the following often called a flow on $A$. The infinitesimal generator $\delta$ of $\alpha$ is a linear map $\delta : D(\delta) \to A$ defined on the dense $*$-algebra $D(\delta) $ of elements $a \in A$ for which the limit 
$$
\delta(a) = \lim_{t \to 0} \frac{\alpha_t(a) - a}{t} \ 
$$
exists. A state $\omega$ on $A$ is a ground state for $\alpha$ when 
\begin{equation}\label{groundstate}
-i \omega(a^*\delta(a)) \geq 0 
\end{equation}
for all $a \in D(\delta)$, and a ceiling state when
\begin{equation}\label{ceilingstate}
i \omega(a^*\delta(a)) \geq 0 
\end{equation}
for all $a \in D(\delta)$, \cite{BR}. Ground states for flows on UHF algebras were introduced and studied in the early days of quantum statistical mechanics; see in particular \cite{Ru}, \cite{PS}, \cite{BR} and \cite{AM} and the references therein. It was soon realized that a ground state need not be unique, see Example 5.3.20 in \cite{BR}, but apparently no one has ever undertaken a study of how big a set of states the ground states can be. It is the purpose of this paper to initiate such a study by handling the flows called generalized gauge actions in \cite{Th2}.


We denote the compact convex set of ground, resp. ceiling states for a flow $\alpha$ by $\GS(\alpha)$, resp. $\CS(\alpha)$. Since generalized gauge actions are approximately inner it follows from \cite{PS} that $\GS(\alpha)$ and $\CS(\alpha)$ are not empty for the flows we consider. In fact, it was shown in \cite{Th1}, and it follows also from the much more general result in \cite{LLN}, that $\GS(\alpha)$ is affinely homeomorphic to the state space of a quotient $C^*$-algebra of the fixed point algebra when $\alpha$ is a generalized gauge action on an AF algebra. This sub-quotient whose state space is a copy of $\GS(\alpha)$ is an AF algebra and the first question concerning the structure of $\GS(\alpha)$ is therefore which AF algebras can occur here when $\alpha$ is a generalized gauge action on a UHF algebra. There is also an AF algebra whose state space is a copy of $\CS(\alpha)$ and the same question therefore applies to that. It was shown in \cite{Th2} that the field of KMS states can be almost arbitrary for a generalized gauge action on a UHF algebra, and it is natural to wonder if there is a relation between the field of KMS states and the sets of ground and ceiling states for such an action. Our main result is the following theorem showing that in general there are no relation, and that the state spaces of any pair of AF algebras can occur as the sets of ground and ceiling states. In the formulation of the theorem the Choquet simplex of $\beta$-KMS states for a flow $\alpha$ is denoted by $\KMS_\beta(\alpha)$ and the state space of a $C^*$-algebra $A$ is denoted by $S(A)$. Also it should be noted that while we allow AF algebras to be finite finite dimensional, a UHF algebra is always infinite dimensional.

\begin{thm}\label{MAIN0} Let $\gamma$ be a generalized gauge action on an $\AF$ algebra and let $U$ be a $\UHF$ algebra. Let $A_+$ and $A_-$ be $\AF$ algebras. There is a generalized gauge action $\alpha$ on $U$ such that 
\begin{itemize}
\item $\KMS_\beta(\alpha)$ is strongly affinely isomorphic to $\KMS_\beta(\gamma)$ for all $\beta \neq 0$,
\item  $\GS(\alpha)$ is affinely homeomorphic to $S(A_+)$ and
\item $\CS(\alpha)$  is affinely homeomorphic to $S(A_-)$.
\end{itemize}
\end{thm}

Combined with Theorem 1.1 in \cite{Th2} it follows that the extreme variation with $\beta$ of the simplexes of $\beta$-KMS states which occur for certain generalized gauge actions on UHF algebras can be extended to $+\infty$ and $-\infty$. More precisely:

\begin{cor}\label{CAR2} Let $U$ be a $\UHF$ algebra, and $A_+$ and $A_-$ two $\AF$ algebras. There is a generalized gauge action $\alpha$ on $U$ such that 
\begin{itemize}
\item for all $\beta \neq 0$ the simplex $\KMS_\beta(\alpha)$ is an infinite dimensional Bauer simplex and $\KMS_\beta(\alpha)$ is not affinely homeomorphic to $\KMS_{\beta'}(\alpha)$ when $\beta \neq \beta'$, 
\item  $\GS(\alpha)$ is affinely homeomorphic to $S(A_+)$ and
\item $\CS(\alpha)$  is affinely homeomorphic to $S(A_-)$.
\end{itemize}
\end{cor}

The notion of strong affine isomorphism occurring in Theorem \ref{MAIN0} was introduced in \cite{Th2} and means that there is an affine bijection between the simplexes whose restriction to the extreme boundaries is a homeomorphism.

In \cite{CM} Connes and Marcolli introduced $\KMS_\infty$-states as the states that are weak* limits of $\beta$-KMS states as $\beta$ tends to infinity. Such states are ground states but they constitute often a considerably smaller set. In fact, for the actions we consider here it follows from \cite{Th1}, and again also in the more general setting of \cite{LLN}, that the $\KMS_\infty$-states correspond to trace states when the ground states are identified with the state space of the sub-quotient mentioned above. As we show by example it is in general not all trace states of the sub-quotient that arise from $\KMS_\infty$-states unless one weakens this notion. Indeed, we show that the tracial state space of the sub-quotient is affinely homeomorphic to the set of states whose restriction to the finite-dimensional $C^*$-subalgebras defined by the underlying Bratteli diagram are $\KMS_\infty$-states for the restriction of the given flow. We call these states local $\KMS_\infty$-states and it follows that they constitute a compact Choquet simplex, which by the results we obtain here can be completely arbitrary. What remains to be figured out is the exact distinction between local and genuine $\KMS_\infty$-states; a problem we leave open. It seems to be difficult, in many cases, to fully unravel the structure of the $\KMS_\infty$-states. Note, for example, that it is not even clear if the set of $\KMS_\infty$-states always constitute a convex set of states; not in general and also not for the flows studied in this paper.

\bigskip

\emph{Acknowledgement} I am grateful to Johannes Christensen for discussions, and for reading and commenting on earlier versions of the paper. The work was supported by the DFF-Research Project 2 `Automorphisms and Invariants of Operator Algebras', no. 7014-00145B.



\section{Generalized gauge actions on AF $C^*$-algebras}

The approximately finite-dimensional $C^*$-algebras, or $\AF$ algebras, can be introduced in several ways. For the present purposes it is natural to consider them as a special case of $C^*$-algebras associated to directed graphs, \cite{KPRR}.

Let $\Gamma$ be a countable directed graph with vertex set $\Gamma_V$ and arrow set
$\Gamma_{Ar}$. For an arrow $a \in \Gamma_{Ar}$ we denote by $s(a) \in \Gamma_V$ its source and by
$r(a) \in \Gamma_V$ its range. In this paper we consider only graphs that are row-finite without sinks, meaning that every vertex admits at least one and at most finitely many arrows, i.e.
$1 \leq \# s^{-1}(v) < \infty$
for all vertexes $v$. An infinite path in $\Gamma$ is an element
$p = (p_i)_{i=1}^{\infty}  \in \left(\Gamma_{Ar}\right)^{\mathbb N}$ such that $r(p_i) = s(p_{i+1})$ for all
$i$.  A finite path $\mu = a_1a_2 \cdots a_n = (a_i)_{i=1}^n \in \left(\Gamma_{Ar}\right)^n$ is
defined similarly. The number of arrows in $\mu$ is its \emph{length}
and we denote it by $|\mu|$. A vertex $v \in \Gamma_V$ will be considered as
a finite path of length $0$. The range and source maps, $s$ and $r$, extend to the set of finite paths in the natural way. The $C^*$-algebra $C^*(\Gamma)$ of the graph $\Gamma$ was introduced in this generality in \cite{KPRR} and it is the universal
$C^*$-algebra generated by a collection $S_a, a \in \Gamma_{Ar}$, of partial
isometries and a collection $P_v, v \in \Gamma_V$, of mutually orthogonal projections subject
to the conditions that
\begin{enumerate}
\item[1)] $S^*_aS_a = P_{r(a)}, \ \forall a \in \Gamma_{Ar}$, and
\item[2)] $P_v = \sum_{a  \in s^{-1}(v)} S_aS_a^*, \ \forall v \in \Gamma_V$.
\end{enumerate} 
For a finite path $\mu = (a_i)_{i=1}^{|\mu|}$ of positive length we set
$$
S_{\mu} = S_{a_1}S_{a_2}S_{a_3} \cdots S_{a_{|\mu|}} \ ,
$$
while $S_{\mu} = P_v$ when $\mu$ is the vertex $v$. The elements $S_{\mu}S_{\nu}^*$, where $\mu,\nu$ are finite paths, span a dense $*$-subalgebra in $C^*(\Gamma)$.

A function $F :  \Gamma_{Ar} \to \mathbb R$ will be called a potential on $\Gamma$ in the following. Using it we can define a continuous one-parameter group $\alpha^F = \left(\alpha^F_t\right)_{t \in \mathbb R}$  on $C^*(\Gamma)$ such that \label{alphaF}
$$
\alpha^F_t(S_a) = e^{i F(a) t} S_a
$$
for all $a \in \Gamma_{Ar}$ and 
$$
\alpha^F_t(P_v) = P_v
$$
for all $v \in \Gamma_V$. Such an action is called a generalized gauge action; the gauge action itself being the one-parameter group corresponding to the constant potential $F =1$.

A Bratteli diagram $\Br$, as introduced by Bratteli in \cite{Br}, is a special class of row-finite directed graphs without sinks in which the vertex set $\Br_V$ is partitioned into level sets,
$$
\Br_V = \sqcup_{n=0}^{\infty} \Br_{n} \ ,
$$
where the number of vertexes in the $n$'th level $\Br_{n}$ is finite, $\Br_{0}$ consists of a single vertex $v_0$ and the arrows emitted from $\Br_n$ end in $\Br_{n+1}$, i.e. $r\left(s^{-1}(\Br_n)\right) \subseteq \Br_{n+1}$ for all $n$. Also, as is customary, we assume that $v_0$ is the only source in $\Br$, i.e. $r^{-1}(v)\cap \Br_{Ar} \neq \emptyset$ for all $v \in \Br_V \backslash \{v_0\}$. The $\AF$ $C^*$-algebra $\AF(\Br)$ which was associated to $\Br$ by Bratteli in \cite{Br} is isomorphic to the corner $P_{v_0}C^*(\Br)P_{v_0}$, as can be seen in the following way. Let $\mathcal P_n$ denote the set of finite paths $\mu$ in $\Br$ emitted from $v_0$ and of length $n$, i.e. $s(\mu) = v_0$ and $|\mu| = n$. Let
$\mathcal P^{(2)}_n = \left\{ (\mu,\mu') \in \mathcal P_n \times \mathcal P_n : \ r(\mu) = r(\mu') \right\}$ and set
$$
E^n_{\mu,\mu'} = S_{\mu}S_{\mu'}^*  \ 
$$ 
when $(\mu,\mu') \in \mathcal P^{(2)}_n$. Then 
\begin{equation}\label{22-09-18}
\{E^n_{\mu,\mu'}\}_{(\mu,\mu') \in \mathcal P^{(2)}_n}
\end{equation}
are matrix units in $P_{v_0}C^*(\Br)P_{v_0}$ and $\sum_{\mu \in \mathcal P_n} E^n_{\mu,\mu} = P_{v_0}$. Let $\AF_n(\Br)$ be the $C^*$-subalgebra of $P_{v_0}C^*(\Br)P_{v_0}$ spanned by the matrix units \eqref{22-09-18}. Then $\AF_n(\Br) \subseteq \AF_{n+1}(\Br)$ and
$$
P_{v_0}C^*(\Br)P_{v_0} = \overline{\bigcup_n \AF_n(\Br)} \ .
$$
Since the Bratteli diagram of the tower 
$\mathbb C \subseteq \AF_1(\Br) \subseteq  \AF_2(\Br) \subseteq \AF_3(\Br) \subseteq \cdots $ is $\Br$, it follows from \cite{Br} that $\AF(\Br) \simeq P_{v_0}C^*(\Br)P_{v_0}$.

Let $ F : \Br_{Ar} \to \mathbb R$ be a potential and $\alpha^F$ the corresponding generalized gauge action on $C^*(\Br)$. Extend $F$ to finite paths $\mu = a_1a_2\cdots a_{|\mu|} \in \left(\Br_{Ar}\right)^{|\mu|}$ of positive length such that
$$
F(\mu) = \sum_{i=1}^{|\mu|} F(a_i) \ .
$$
Then
\begin{equation*}\label{22-09-18e}
\alpha^F_t(E^n_{\mu, \mu'}) = e^{it(F(\mu) - F(\mu'))} E^n_{\mu, \mu'} \ = \Ad e^{it H_n}\left(E^n_{\mu,\mu'}\right) \ ,
\end{equation*}
where 
$$
H_n = \sum_{\mu \in \mathcal P_n} F(\mu) E^n_{\mu,\mu} \in \AF_n(\Br) \ .
$$
By restriction the generalized gauge action $\alpha^F$ gives rise to a continuous one-parameter group of automorphisms on $\AF(\Br)$ which we also denote by $\alpha^F$ and call it a generalized gauge action on $\AF(\Br)$. The derivation $\delta_F$ generating the flow $\alpha^F$ on $\AF(\Br)$ is given by the formula
\begin{equation}\label{29-08-20}
\delta_F(a) = i(H_na-aH_n) 
\end{equation}
when $a \in \AF_n(\Br)$.

 We let $P(\Br)$ denote the set of infinite paths $p =p_1p_2p_3\cdots $ in $\Br$ that are emitted from the top vertex, i.e. $s(p_1) = v_0$. For each $n $ and $p \in P(\Br)$ we denote by $p[1,n]$ the finite path in $\Br$ consisting of the first $n$ arrows in $p$, i.e. $p[1,n] = p_1p_2\cdots p_n$. An infinite path $p \in P(\Br)$ is called a geodesic, or an $F$-geodesic if it is necessary to specify the potential, when
$$
F\left(p[1,n]\right) \ \leq \ F(\mu) 
$$
for all $\mu \in \mathcal P_n$ with $r(\mu) = r(p_n)$ and for all $n \in \mathbb N$. Let $\Geo^F(\Br)$ denote the set of geodesics in $P(\Br)$. It is easy to see that $\Geo^F(\Br)$ is never empty. The set of geodesics determine a Bratteli sub-diagram $\Br^+$ of $\Br$ in the following way. For $n \geq 1$ the set of vertexes in $\Br^+_n$ is
$$
\Br^+_n \ = \ \left\{ v \in \Br_n : \ v = r\left(p[1,n]\right) \ \text{for some} \ p \in \Geo^F(\Br) \right\} \ 
$$   
while $\Br^+_0 = \{v_0\}$; the top vertex in $\Br$. The arrows in $\Br^+_{Ar}$ are the arrows $a \in \Br_{Ar}$ with the properties that $a = p_n$ for some geodesic $p = p_1p_2p_3 \cdots \ \in \Geo^F(\Br)$.

\begin{lemma}\label{09-05-19} $P\left(\Br^+\right) \ = \ \Geo^F(\Br)$.
\end{lemma}
\begin{proof} The inclusion $\Geo^F(\Br) \subseteq P\left(\Br^+\right)$ is obvious and it suffices therefore to consider an infinite path $p \in P\left(\Br^+\right)$ and show that $F(p[1,n]) \leq F(\mu)$ for all $\mu \in \mathcal P_n$ with $r(\mu) = r(p_n)$. To deduce this by induction in $n$ note that it is clearly true when $n=1$. Assume therefore that it holds for $n-1$. By definition of $\Br^+$ there is a $q \in \Geo^F(\Br)$ such that $q_n = p_n$. Since $r\left(p[1,n-1]\right) = r\left(q[1,n-1]\right)$ it follows from the induction hypothesis that $F\left(q[1,n-1]\right) = F\left(p[1,n-1]\right)$ and hence $F\left(q[1,n]\right) = F\left(p[1,n]\right)$ since $p_n = q_n$. Therefore $F\left(p[1,n]\right) = F\left(q[1,n]\right) \leq F(\mu)$ for all $\mu \in \mathcal P_n$ with $r(\mu) = r(p_n)$.
\end{proof}

We denote by $\AF(\Br)^{\alpha^F}$ the fixed point algebra of $\alpha^F$. In order to relate $\AF(\Br)$ and $\AF(\Br^+)$ we use a conditional expectation $\AF(\Br) \ \to \ \AF(\Br)^{\alpha^F}$.

\begin{lemma}\label{15-05-19} The fixed point algebra $\AF(\Br)^{\alpha^F}$ of $\alpha^F$ is the range of a conditional expectation 
$R: \AF(\Br) \to \AF(\Br)^{\alpha^F}$ 
defined such that
\begin{equation}\label{15-05-19}
R(a) \ = \ \lim_{T \to \infty} \frac{1}{T} \int_0^T \alpha^F_t(a) \ \mathrm{d}t \ .
\end{equation}
\end{lemma}
\begin{proof} Note that $ a \mapsto \frac{1}{T}\int_0^T \alpha^F_t(a) \ \mathrm{d}t$ is a positive contraction for all $T > 0$. Since
$$
\lim_{T \to \infty} \frac{1}{T} \int_0^T \alpha^F_t(E^n_{\mu,\mu'}) \ \mathrm{d}t \ = \ \begin{cases} 0 \  & \ \text{when} \ F(\mu) \neq F(\mu')\\ E^n_{\mu,\mu'} \ & \ \text{when} \ F(\mu) = F(\mu') \ , \end{cases}
$$
it follows therefore that the limit \eqref{15-05-19} exists for all $a \in \AF(\Br)$. The resulting operator $R$ is easily seen to be a conditional expectation onto $\AF(\Br)^{\alpha^F}$.
\end{proof}

Note that 
$$
R(\AF_n(\Br)) \ = \ \AF_n(\Br) \cap  \AF(\Br)^{\alpha^F} \ = \ \Span \{ E^n_{\mu,\mu'} : \ F(\mu) = F(\mu') \}
$$
is the fixed point algebra $\AF_n(\Br)^{\alpha^F}$ of the flow $\Ad e^{itH_n}$ on $\AF_n(\Br)$, and that
$$
\overline{\bigcup_n \AF_n(\Br)^{\alpha^F}} \ = \  \AF(\Br)^{\alpha^F} \ .
$$
Let $\mathcal G_n$ be the set of paths in $\Br^+$ of length $n$ emitted from $v_0$ and set
$$
Q_n = \sum_{\mu \in \mathcal G_n} E^n_{\mu,\mu} \ \in \ \AF_n(\Br) \ .
$$ 
Then $\AF_n(\Br^+)$ is the corner in $\AF_n(\Br)$ defined by $Q_n$, i.e. 
$$
\AF_n(\Br^+) \ = \  \Span \left\{ E^n_{\mu,\mu'}: \ \mu,\mu' \in \mathcal G_n \right\} \ = \ Q_n\AF_n(\Br)Q_n \ .
$$
The map 
$$
q_n(a) \ = \ Q_naQ_n
$$ 
is a positive contraction $q_n :\AF_n(\Br) \to \AF_n(\Br^+)$  whose restriction to $\AF_n(\Br)^{\alpha^F}$ is a unital and surjective $*$-homomorphism since $Q_n$ is in the center of $\AF_n(\Br)^{\alpha^F}$. In addition the diagram
\begin{equation}\label{26-11-18c}
\begin{xymatrix}{
\AF_n(\Br)^{\alpha^F} \ar[d]_{q_n}& \subseteq & \AF_{n+1}(\Br)^{\alpha^F} \ar[d]^{q_{n+1}} \\
\AF_n(\Br^+)  & \subseteq & \AF_{n+1}(\Br^+)  }
  \end{xymatrix}
\end{equation}
commutes. It follows that there is a unital surjective $*$-homomorphism 
$$
q_F : \AF(\Br)^{\alpha^F} \ \to \ \AF(\Br^+)
$$
such that $q_F(a) = Q_naQ_n$ when $a \in \AF_n(\Br)^{\alpha^F}$. Set
$$
Q_F = q_F \circ R \ ,
$$
which is a surjective positive linear map from $\AF(\Br)$ onto $\AF(\Br^+)$.

\begin{prop}\label{09-05-19b} Let $\omega \in S\left(\AF(\Br)\right)$. The following conditions are equivalent:
\begin{enumerate}
\item $\omega$ is a ground state for $\alpha^F$.
\item $\omega(Q_n) = 1$ for all $n \in \mathbb N$.
\item $\omega$ factorizes through $Q_F$, i.e. there is a state $\omega' \in S\left(\AF\left(\Br^+\right)\right)$ such that $\omega \ = \ \omega' \circ Q_F$.
\end{enumerate}
\end{prop}
\begin{proof} (1) $\Rightarrow$ (2): Note that $1 - Q_n$ is the sum of the projections $E^n_{\nu,\nu}$ where $\nu \in \mathcal P_n \backslash \mathcal G_n$. Let $\nu \in \mathcal P_n \backslash \mathcal G_n$. It suffices to show that $\omega\left(E^n_{\nu,\nu}\right) = 0$. To this end note that since $\nu \notin \mathcal G_n$ there is a $k > n$ such that the image of $E^n_{\nu,\nu}$ in $\AF_k(\Br)$ is a sum
$$
E^n_{\nu,\nu} \ = \ \sum_{\nu' \in A} E^k_{\nu',\nu'}
$$
with the property that for each $\nu' \in A$ there is a path $\mu \in \mathcal P_k$ with $r(\mu) = r(\nu')$ for which $F(\nu') > F(\mu)$. It follows from \eqref{29-08-20} that
$$
\delta_F\left( E^k_{\mu,\nu'}\right) \ = \ i\left(F(\mu)-F(\nu')\right) E^k_{\mu,\nu'} \ . 
$$
Since $\omega$ is a ground state we have that
$$
0 \ \leq \ -i\omega\left( {E^k_{\mu,\nu'}}^* \delta_F\left( E^k_{\mu,\nu'}\right)\right) \ = \ \omega\left(E^k_{\nu',\nu'}\right)\left( F(\mu) - F(\nu')\right) \ ,
$$
which implies that $\omega\left(E^k_{\nu',\nu'}\right) = 0$. It follows that
$$
\omega\left(E^n_{\nu,\nu} \right) \ = \ \sum_{\nu' \in A} \omega\left(E^k_{\nu',\nu'} \right) = 0 \ .
$$
(2) $\Rightarrow$ (3): It follows from (2) that $\omega (a) \ = \ \omega(Q_naQ_n)$ for all $a \in \AF_n(\Br)$ and then from the commutativity of the diagram \eqref{26-11-18c} that we can define a state $\omega'$ on $\AF(\Br^+)$ such that
$$
\omega'|_{\AF_n(\Br^+)} = \omega|_{\AF_n(\Br^+)} \ 
$$
for all $n$. Then $\omega' \circ Q_F(a) \ = \ \omega'(Q_naQ_n) \ = \ \omega(Q_naQ_n) \ = \ \omega(a)$ when $a \in \AF_n(\Br)$, and hence $\omega' \circ Q_F \ = \ \omega$.

(3) $\Rightarrow$ (1): Note that $\bigcup_n \AF_n(\Br)$ is a core for $\delta_F$ by Corollary 3.1.7 in \cite{BR}. To show that $\omega$ is a ground state it suffices therefore to show that $-iQ_F(a^*\delta_F(a)) \geq 0$ when $a \in \AF_n(\Br)$. To this end we write 
$$
a \ = \sum_{(\mu, \mu') \in \mathcal P_n^{(2)}}\lambda_{\mu,\mu'}E^n_{\mu,\mu'} \ ,
$$
where $\lambda_{\mu,\mu'} \in \mathbb C$. For $\mu \in \mathcal P_n$ set
$$
b_{\mu} \ = \sum_{\nu' \in \mathcal G_n, r(\nu') = r(\mu)} \sqrt{F(\mu) - F(\nu')}\lambda_{\mu,\nu'} E^n_{\mu,\nu'} \ ,
$$ 
with the convention that a sum over the empty set is zero. Using that
$$
Q_F(E^n_{\mu,\mu'}) = \begin{cases} E^n_{\mu,\mu'} & \ \text{when} \ \mu,\mu' \in \mathcal G_n, \\ 0 & \ \text{otherwise } \ , \end{cases} 
$$
we find
\begin{equation*}
\begin{split}
&-iQ_F\left(a^*\delta_F(a)\right) \\ 
&\\ 
& = \ \sum_{(\mu,\mu'),(\nu,\nu') \in \mathcal P^{(2)}_n} (F(\nu)-F(\nu'))\overline{\lambda_{\mu,\mu'}}\lambda_{\nu,\nu'} Q_F \left(E^n_{\mu',\mu}E^n_{\nu,\nu'}\right) \\
&\\
& = \ \sum_{(\mu,\mu'),(\nu,\nu') \in \mathcal P^{(2)}_n}   \delta_{\mu,\nu}(F(\nu)-F(\nu'))\overline{\lambda_{\mu,\mu'}}\lambda_{\nu,\nu'} Q_F\left(E^n_{\mu',\nu'}\right) \\
&\\
& = \ \sum_{\mu \in \mathcal P_n} b_{\mu}^*b_{\mu} \ \geq  \ 0 \ .
\end{split}
\end{equation*}

\end{proof}

\begin{thm}\label{09-05-19d} The set $\GS(\alpha^F)$ of ground states for the generalized gauge action $\alpha^F$ is a closed face in the state space of $\AF(\Br)$ affinely homeomorphic with the state space of $\AF\left(\Br^+\right)$ via the map
\begin{equation}\label{16-09-20}
S\left(\AF(\Br^+)\right) \ni \omega \ \mapsto \ \omega \circ Q_F \ .
\end{equation}
\end{thm}
\begin{proof} That $\GS(\alpha^F)$ is a closed face in $S(\AF(\Br))$ follows from the equivalence of (1) and (2) in Proposition \ref{09-05-19b}. That $\GS(\alpha^F)$ is affinely homeomorphic to $S(\AF(\Br^+))$ via the map \eqref{16-09-20} follows from the equivalence of (1) and (3) in Proposition \ref{09-05-19b} since $Q_F$ is surjective.
\end{proof}

That the set of ground states constitute a closed face in the state space is a general fact and holds for all flows on unital $C^*$-algebras by Theorem 5.3.37 in \cite{BR}.

\begin{example}\label{09-05-19e} A virtue of generalized gauge actions and the description of the set of ground states in Theorem \ref{09-05-19d} is that it is easy to construct examples. To illustrate this observe that the following two Bratteli diagrams, $\Br_1$ and $\Br_2$, are identical and $\AF(\Br_1) = \AF(\Br_2)$ is the CAR algebra. The generalized gauge actions we equip the algebras with are given by the labels on the arrows; the numbers show the value of the potential on the arrow.

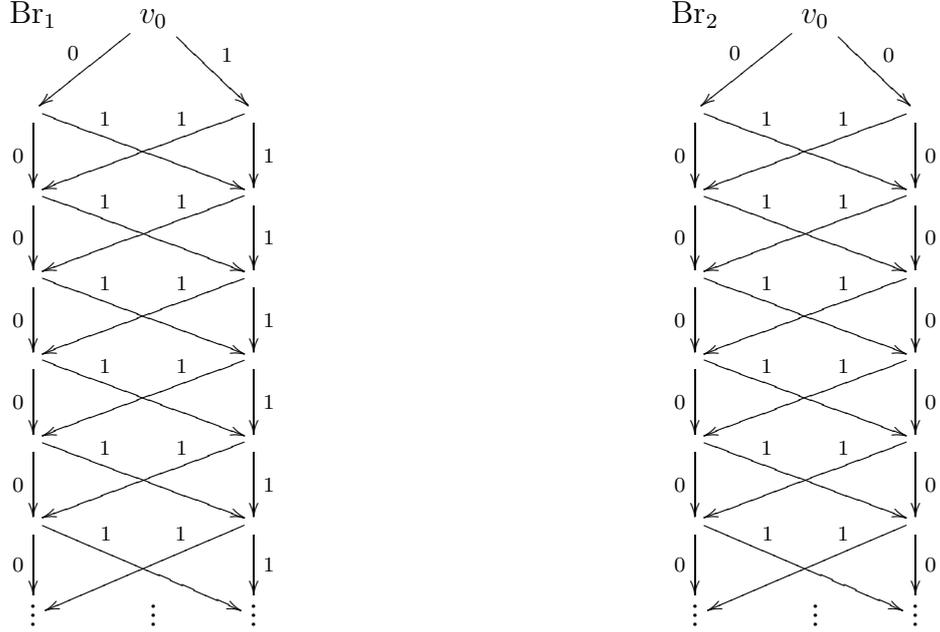
\begin{equation*}\label{11-12-18a*}
\begin{xymatrix}{ & \Br_1 & \ar[ld]_-0 v_0 \ar[rd]^-1 && & & &   &\Br_2 & \ar[ld]_-0 v_0 \ar[rd]^-0 \\
& \ar[d]_-0 \ar[rrd]^(0.3){1}  &  & \ar[d]^-1 \ar[lld]_(0.3){1} &&&& & \ar[d]_-0 \ar[rrd]^(0.3){1} &  & \ar[d]^-0 \ar[lld]_(0.3){1} \\  
& \ar[d]_-0  \ar[rrd]^(0.3){1} &  & \ar[d]^-1 \ar[lld]_(0.3){1} &&&& & \ar[d]_-0  \ar[rrd]^(0.3){1} &  & \ar[d]^-0 \ar[lld]_(0.3){1} \\  
& \ar[d]_-0  \ar[rrd]^(0.3){1} &  & \ar[d]^-1 \ar[lld]_(0.3){1} &&&& & \ar[d]_-0  \ar[rrd]^(0.3){1} &  & \ar[d]^-0 \ar[lld]_(0.3){1} \\  
& \ar[d]_-0  \ar[rrd]^(0.3){1} &  & \ar[d]^-1 \ar[lld]_(0.3){1} &&&& & \ar[d]_-0  \ar[rrd]^(0.3){1} &  & \ar[d]^-0 \ar[lld]_(0.3){1} \\  
& \ar[d]_-0  \ar[rrd]^(0.3){1} &  & \ar[d]^-1 \ar[lld]_(0.3){1} &&&& & \ar[d]_-0  \ar[rrd]^(0.3){1} &  & \ar[d]^-0 \ar[lld]_(0.3){1} \\  
& \ar[d]_-0  \ar[rrd]^(0.3){1} &  & \ar[d]^-1 \ar[lld]_(0.3){1} &&&& & \ar[d]_-0  \ar[rrd]^(0.3){1} &  & \ar[d]^-0 \ar[lld]_(0.3){1} \\  
&\vdots &\vdots  & \vdots   &&&& &\vdots &\vdots  & \vdots }
 \end{xymatrix}
 \end{equation*}

\bigskip 

Let $F_i$ be the potential described by the diagram $\Br_i$. The Bratteli sub-diagrams $\Br^+_1$ and $\Br^+_2$ obtained from the $F_1$- and the $F_2$-geodesics are given by the following sub-diagrams:

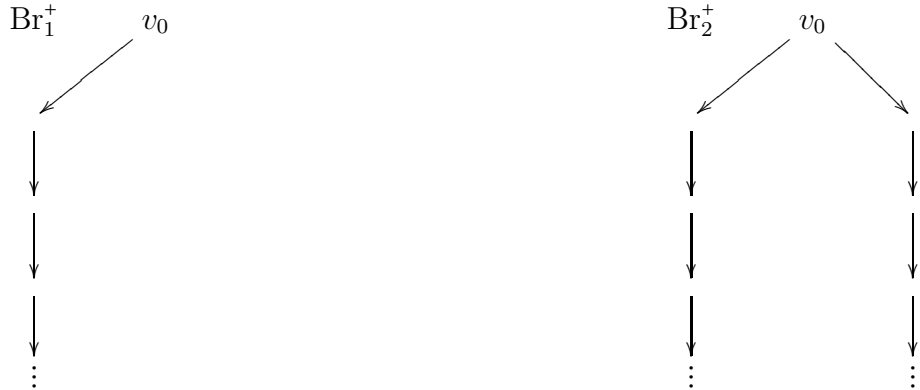
\begin{equation*}\label{11-12-18a}
\begin{xymatrix}{ &\Br^+_1 & \ar[ld] v_0  & &&&&  &\Br^+_2 & \ar[ld] v_0 \ar[rd] \\
& \ar[d]  &  &  &&&&& \ar[d]  &  & \ar[d] \\  
& \ar[d]   &  &  &&&& & \ar[d]   &  & \ar[d] \\  
& \ar[d]   &  &  &&&& & \ar[d]   &  & \ar[d] \\  
&\vdots & &  &&&&&\vdots & & \vdots  }
 \end{xymatrix}
 \end{equation*}
 
 \bigskip

It follows from Theorem \ref{09-05-19d} that there is a unique ground state for $\alpha^{F_1}$ while the set of ground states for $\alpha^{F_2}$ is affinely homeomorphic to the interval $[0,1]$; the state space of $\mathbb C^2$. 
 
\end{example}


\section{The main result}

Let $\beta \in \mathbb R$. A $\beta$-KMS states for a flow $\alpha$ on a unital $C^*$-algebra $A$ is a state $\omega \in S(A)$ such that
$$
\omega(ab) = \omega(b \alpha_{i\beta}(a))
$$
for all elements $a,b$ that are analytic for $\alpha$, or alternatively for all elements $a$ in $A$ and all elements $b$ from a dense $*$-subalgebra of analytic elements for $\alpha$, cf. \cite{BR}. The set of $\beta$-KMS states for $\alpha$ will be denoted by $\KMS_\beta(\alpha)$. For the flow $\alpha^F$ on $\AF(\Br)$, where $\bigcup_n \AF_n(\Br)$ is a dense $*$-subalgebra of $\alpha^F$-analytic elements, a state $\omega$ of $\alpha^F$ is a $\beta$-KMS state iff
\begin{equation}\label{06-08-20a}
\omega\left(E^n_{\mu,\mu'}E^n_{\nu,\nu'}\right) = e^{-\beta(F(\mu) - F(\mu'))} \omega\left(E^n_{\nu,\nu'}E^n_{\mu,\mu'}\right)
\end{equation}
for all $n$ and all $\mu,\mu',\nu,\nu' \in \mathcal P_n$. 

 For the study of $\KMS_\beta(\alpha^F)$ we make use of projective matrix systems as defined in \cite{Th2}. A projective matrix system over $\Br$ is a sequence $A^{(j)}, j =1,2,3,\cdots$, where 
$$
A^{(j)} = \left(A^{(j)}_{v,w}\right)_{(v,w) \in \Br_{j-1} \times \Br_j}
$$
is a non-negative real matrix over $\Br_{j-1} \times \Br_j$ subject to the condition that
\begin{equation}\label{09-05-19h}
\left(A^{(1)}A^{(2)} \cdots A^{(k)}\right)_{v_0,w} \neq 0
\end{equation}
for all $w \in \Br_k$ and all $k \geq 1$. Let $\varprojlim_j A^{(j)}$ be the set of sequences
$$
\left(\psi^j\right)_{j=0}^{\infty} \in \prod_{j=0}^{\infty} [0,\infty)^{\Br_j} 
$$
for which $\psi^{j-1} = A^{(j)}\psi^{j}$ for all $j = 1,2,\cdots$. The set $\varprojlim_j A^{(j)}$ is a locally compact Hausdorff space in the topology inherited from the product topology of $\prod_{j=0}^{\infty} [0,\infty)^{\Br_j} $. 

 When $F$ is a potential defined on $\Br$ and $\beta \in \mathbb R$ a real number we define a projective matrix $\Br(\beta)^{(j)}, j = 1,2,3, \cdots$, over $\Br$ such that
\begin{equation}\label{10-05-19i}
\Br(\beta)^{(j)}_{v,w} \ \ = \sum_{a \in s^{-1}(v) \cap r^{-1}(w)} e^{-\beta F(a)} \ .
\end{equation}
 When $\omega \in \KMS_\beta(\alpha^F)$, define vectors $\psi(\omega)^j  \in [0,\infty)^{\Br_j}, \ j \geq 1$, such that
$$
\psi(\omega)^j_v \ = \ e^{\beta F(\mu)} \omega\left(E^j_{\mu,\mu}\right) \ ,
$$
where $\mu \in \mathcal P_j$ terminates at $v$; i.e. $r(\mu) = v$. The value $\psi(\omega)^j_v$ is independent of which $\mu \in \mathcal P_j$ terminating at $v$ we use; indeed, if $\mu'\in \mathcal P_j$ and $r( \mu') = r(\mu)$, it folllows from \eqref{06-08-20a}that
\begin{align*}
& e^{\beta F(\mu)} \omega\left(E^j_{\mu,\mu}\right) =  e^{\beta F(\mu)} \omega\left(E^j_{\mu,\mu'}E^j_{\mu',\mu}\right) \\
&=  e^{\beta F(\mu)}e^{-\beta(F(\mu)-F(\mu'))} \omega\left(E^j_{\mu',\mu}E^j_{\mu,\mu'}\right) = e^{\beta F(\mu')} \omega\left(E^j_{\mu',\mu'}\right) \ .
\end{align*}
Furthermore, when $v \in \Br_{j-1}$ and $\mu_v \in \mathcal P_{j-1}$ is a path terminating at $v$, $j \geq 2$, we have that 
\begin{align*}
&\sum_{w \in \Br_j} \Br(\beta)^{(j)}_{v,w}\psi(\omega)^j_w \ = \ \sum_{w \in \Br_j}\sum_{a \in s^{-1}(v) \cap r^{-1}(w)} e^{-\beta F(a)} \psi(\omega)^j_w \\
&= \sum_{w \in \Br_j}\sum_{a \in s^{-1}(v) \cap r^{-1}(w)} e^{-\beta F(a)} e^{\beta F(\mu_va)} \omega\left(E^j_{\mu_va,\mu_va}\right) \\
&=  e^{\beta F(\mu_v)} \sum_{b \in s^{-1}(v)}\omega\left(E^j_{\mu_vb,\mu_vb}\right) \ = \ e^{\beta F(\mu_v)}\omega(E^{j-1}_{\mu_v}) = \psi(\omega)^{j-1}_v \ .
\end{align*}
Set $\psi( \omega)^0 = \Br(\beta)^{(1)}\psi(\omega)^1$. Then $\psi(\omega) =\left(\psi(\omega)^j\right)_{j=0}^{\infty}$ is an element of the inverse limit space $\varprojlim_j \Br(\beta)^{(j)}$. As shown in Proposition 3.4 of \cite{Th2} the map $\omega \mapsto \psi(\omega)$ gives rise to a homeomorphism:

\begin{lemma}\label{10-05-19} The map $\KMS_\beta(\alpha^F) \ni \omega \mapsto \left(\psi(\omega)^j\right)_{j=0}^{\infty}$ is an affine homeomorphism onto
$$
\left\{ \left(\psi^j\right)_{j=0}^{\infty}   \in \varprojlim_j \Br(\beta)^{(j)} : \ \psi^0 = 1\right\} \ .
$$
\end{lemma}

This lemma makes it possible to control the KMS-states when manipulating Bratteli diagrams and generalized gauge actions. For such purposes we need some preparations.

\begin{lemma}\label{22-11-18a} Let $\Br$ be a Bratteli diagram and $\{ A^{(j)}\}$ a projective matrix system over $\Br$. Let $\{m_j\}_{j=1}^{\infty}$ be a sequence of positive numbers and set $B^{(j)} = m_jA^{(j)}$. Then $\varprojlim_j B^{(j)}$ is affinely homeomorphic to $\varprojlim_j A^{(j)}$.
\end{lemma}
\begin{proof} Define $\Phi : \varprojlim_j B^{(j)} \ \to \ \varprojlim_j A^{(j)}$ such that
$$
\Phi\left((x_i)_{i=0}^{\infty}\right) = \left(x_0,\ m_1x_1, \ m_2m_1x_2,\ m_3m_2m_1 x_3, \ \cdots \right) \ .
$$
\end{proof}

The following is a variant of Lemma 5.2 in \cite{Th2} which is needed for the present purposes. The norms used in the statement and proof are the Euclidean norms (or $l^2$-norms) on vectors, and on matrices it is the corresponding operator norm.

\begin{lemma}\label{09-05-19g} Let $\Br$ be a Bratteli diagram and $\{ A^{(j)}\}$ a projective matrix system over $\Br$. Let $\left\{\epsilon_j\right\}_{j=1}^\infty$ be a sequence of positive numbers in the interval $\left]0,\frac{1}{2}\right[$ such that
\begin{equation}\label{07-08-20} \epsilon_j \sqrt{\# \Br_j}  \left(\left(A^{(1)}A^{(2)} \cdots A^{(j)}\right)_{v_0,w} \right)^{-1} \prod_{k=1}^j \left(2\left\|A^{(k)}\right\| +1 \right) \ \leq \ 4^{-j}
\end{equation}
 for all $w \in \Br_j$. When $\{ B^{(j)}\}$ is another projective matrix system over $\Br$ such that for some $N \in \mathbb N$,
\begin{equation}\label{09-05-19k}
\left| A^{(j)}_{v,w} - B^{(j)}_{v,w}\right| \ \leq \ \epsilon_j A^{(j)}_{v,w} \ \ \ \forall v \in \Br_{j-1}, \ \forall w \in \Br_j,
\end{equation}
when $j \geq N$, then $\varprojlim_j A^{(j)}$ is affinely homeomorphic to $\varprojlim_j B^{(j)}$.

\end{lemma}
\begin{proof} The proof is essentially the same as the proof of Lemma 5.2 in \cite{Th2}, but the assumptions are sufficiently different to justify that we present the details. Note that it follows from \eqref{09-05-19k} that
$$
B^{(j)}_{v,w} \ \geq \ (1-\epsilon_j)A^{(j)}_{v,w} \ \geq \  \frac{1}{2}A^{(j)}_{v,w}
$$
for all $v\in \Br_{j-1},w \in \Br_j$ when $j \geq N$. There is therefore a positive number $K_N$, depending on $A^{(i)}$ and $B^{(i)}$ for $ i=1,2,\cdots, N$, such that 
$$
\left(\left(B^{(1)}B^{(2)} \cdots B^{(j)}\right)_{v_0,w} \right)^{-1} \leq K_N 2^j \left(\left(A^{(1)}A^{(2)} \cdots A^{(j)}\right)_{v_0,w} \right)^{-1}
$$
for all $w \in \Br_j$ when $j \geq N$. Let $\phi \in \varprojlim_j B^{(j)}$. Then
\begin{equation*}\label{06-08-20b}
\begin{split}
&\left\|\phi^j\right\| \leq \sqrt{\# \Br_j}  \max_{w \in \Br_j}  \left(\left(B^{(1)}B^{(2)} \cdots B^{(j)}\right)_{v_0,w} \right)^{-1} \phi^0 \\
&\leq K_N2^j  \sqrt{\# \Br_j}\max_{w \in \Br_j} \left(\left(A^{(1)}A^{(2)} \cdots A^{(j)}\right)_{v_0,w} \right)^{-1} \phi^0 \ 
\end{split}
\end{equation*}
for all $j \geq N$. Note that \eqref{09-05-19k} implies that $\left\|A^{(j)} - B^{(j)}\right\| \leq \epsilon_j \left\|A^{(j)}\right\|$. Hence, when $k,j\in \mathbb N$ and $j\geq N$, we find that
\begin{equation}\label{20-09-18dbxx}
\begin{split}
& \left\|A^{(j)}A^{(j+1)} \cdots A^{(j+k)}\phi^{j+k} - A^{(j)}A^{(j+1)} \cdots A^{(j+k+1)}\phi^{j+k+1} \right\| \\
& = \left\|A^{(j)}A^{(j+1)} \cdots A^{(j+k)}B^{(j+k+1)}\phi^{j+k+1} - A^{(j)}A^{(j+1)} \cdots A^{(j+k+1)}\phi^{j+k+1} \right\| \\
& \leq \left\| A^{(j)}A^{(j+1)} \cdots A^{(j+k)}\right\| \left\| B^{(j+k+1)}\phi^{j+k+1}  - A^{(j+k+1)}\phi^{j+k+1}\right\| \\
& \leq \left(\prod_{n=j}^{j+k}\left\|A^{(n)}\right\|\right)  \left\| B^{(j+k+1)}\phi^{j+k+1}  - A^{(j+k+1)}\phi^{j+k+1}\right\| \\
& \leq \left(\prod_{n=j}^{j+k}\left\|A^{(n)}\right\|\right)\left\|A^{(j+k+1)}\right\| \left\|\phi^{j+k+1}\right\|\epsilon_{j+k+1} \\
& \leq \left(\prod_{n=j}^{j+k+1}\left\|A^{(n)}\right\|\right)K_N 2^{j+k+1}\sqrt{\# \Br_{j+k+1}} \max_{w \in \Br_{j+k+1}}  \left(\left(A^{(1)} \cdots A^{(j+k+1)}\right)_{v_0,w} \right)^{-1} \phi^0 \epsilon_{j+k+1} \\
& \leq K_N 2^{-j-k-1} \phi^0 \ . 
\end{split}
\end{equation}
It follows that
$$
\left\{ A^{(j)}A^{(j+1)} \cdots A^{(j+k)}\phi^{j+k}\right\}_{k=1}^{\infty}
$$
is a Cauchy-sequence in $\mathbb R^{\Br_{j-1}}$ for all $j \in \mathbb N$ and we set
$$
\psi^{j-1} = \lim_{k \to \infty} A^{(j)}A^{(j+1)} \cdots A^{(j+k)}\phi^{j+k} \ ,
$$
$j = 1,2,3, \cdots $. Note that  $\psi = (\psi^j)_{j=0}^{\infty} \in \varprojlim_j A^{(j)}$. We define an affine map $S : \varprojlim_j B^{(j)} \to \varprojlim_j A^{(j)}$ such that $S\phi = \psi$.  It follows from \eqref{20-09-18dbxx} that
$$
\left\|\psi^{j-1} -  A^{(j)}A^{(j+1)} \cdots A^{(j+k)}\phi^{j+k}\right\| \leq K_N \phi^0\sum_{n=j+k+1}^{\infty} 2^{-n} \ 
$$
for all $j \geq N$ and all $k$, implying that $S$ is continuous.

We proceed in a similar way to construct an inverse for $S$. Let $\psi = (\psi^j)_{j=0}^{\infty} \in \varprojlim_j A^{(j)}$. It follows from \eqref{09-05-19k} that $\left\|B^{(j)}\right\| \leq (1+\epsilon_j)\left\|A^{(j)}\right\| \leq 2\left\|A^{(j)}\right\|$ when $j \geq N$. Then, for $j,k\in \mathbb N,\ j \geq N$, estimates similar to the preceding now yields that
\begin{equation}\label{20-09-18db}
\begin{split}
& \left\|B^{(j)}B^{(j+1)} \cdots B^{(j+k)}\psi^{j+k} - B^{(j)}B^{(j+1)} \cdots B^{(j+k+1)}\psi^{j+k+1} \right\| \\
&\\
&  \leq \ \left(\prod_{n=j}^{j+k+1} 2\|A^{(n)}\|\right) \sqrt{\# \Br_{j+k+1}}  \max_{w \in \Br_{j+k+1}}  \left(\left(A^{(1)} \cdots A^{(j+k+1)}\right)_{v_0,w} \right)^{-1} \psi^0 \epsilon_{j+k+1}\\
&\\
& \leq \ \psi^0 4^{-j-k-1} \ .
\end{split}
\end{equation}
It follows that $\left\{ B^{(j)}B^{(j+1)} \cdots B^{(j+k)}\psi^{j+k}\right\}_{k=1}^{\infty}$
is a Cauchy-sequence in $\mathbb R^{\Br_{j-1}}$ for all $j$ and we set 
$$
\phi^{j-1} = \lim_{k \to \infty} B^{(j)}B^{(j+1)} \cdots B^{(j+k)}\psi^{j+k} \ ,
$$
$j = 1,2,3, \cdots $. Then $\phi = (\phi^j)_{j=0}^{\infty} \in \varprojlim_j B^{(j)}$ and the assignment $T \psi = \phi$ gives us an affine map $T : \varprojlim_j A^{(j)} \to \varprojlim_j B^{(j)}$. It follows from \eqref{20-09-18db} that
$$
\left\|\phi^{j-1} -  B^{(j)}B^{(j+1)} \cdots B^{(j+k)}\psi^{j+k}\right\| \leq\psi^0\sum_{n=j+k+1}^{\infty} 4^{-n} \ 
$$
for all $j \geq N$ and $k \in \mathbb N$, implying that $T$ is continuous. 

To see that $S \circ T$ is the identity on $\varprojlim_j A^{(j)}$, let $\psi \in \varprojlim_j A^{(j)}$ and consider $j,k, m \in \mathbb N$, $j \geq N$. Then
\begin{equation*}
\begin{split}
&\left\| A^{(j)}\cdots A^{(j+k)}B^{(j+k+1)}\cdots B^{(j+k+m)}\psi^{j + k +m} - \psi^{j-1} \right\| \\
&\\
& = \left\| A^{(j)}\cdots A^{(j+k)}B^{(j+k+1)}\cdots B^{(j+k+m)}\psi^{j + k +m} - A^{(j)}A^{(j+1)}\cdots A^{(j+k+m)}\psi^{j + k +m} \right\| \\
&\\
&\leq \left(\prod_{n=j}^{j+k} \left\|A^{(n)}\right\|\right)\left\| B^{(j+k+1)}\cdots B^{(j+k+m)}\psi^{j + k +m} - \psi^{j+k} \right\| \\
&\\
& \leq \left(\prod_{n=j}^{j+k} \left\|A^{(n)}\right\|\right)\left\| B^{(j+k+1)}\cdots B^{(j+k+m-1)}\left(B^{(j+k+m)}\psi^{j + k +m} -  A^{(j+k+m)}\psi^{j + k +m}\right) \right\| \\
& \ \ \ \ \ \ \ \ \ \ \ \ \ \ \ \ + \left(\prod_{n=j}^{j+k} \left\|A^{(n)}\right\|\right) \left\| B^{(j+k+1)}\cdots B^{(j+k+m-1)}A^{(j+k+m)}\psi^{j + k +m} - \psi^{j+k} \right\|\\
&\\
&  = \left(\prod_{n=j}^{j+k} \left\|A^{(n)}\right\|\right)\left( \prod_{n=j+k+1}^{j+k+m-1} \left\|B^{(n)}\right\|\right) \left\| B^{(j+k+m)}\psi^{j + k +m} -  A^{(j+k+m)}\psi^{j + k +m} \right\| \\
& \ \ \ \ \ \ \ \ \ \ \ \ \ \ \ \ + \prod_{n=j}^{j+k} \left\|A^{(n)}\right\| \left\| B^{(j+k+1)}\cdots B^{(j+k+m-1)}\psi^{j + k +m-1} - \psi^{j+k} \right\|\\
&\\
&  \leq \left(\prod_{n=j}^{j+k} \left\|A^{(n)}\right\|\right) \left(\prod_{n=j+k+1}^{j+k+m-1} \left\|B^{(n)}\right\|\right) \left\|A^{(j+k+m)}\right\|  \left\| \psi^{j + k +m} \right\|\epsilon_{j+k+m} \\
& \ \ \ \ \ \ \ \ \ \ \ \ \ \ \ \ + \prod_{n=j}^{j+k} \left\|A^{(n)}\right\| \left\| B^{(j+k+1)}\cdots B^{(j+k+m-1)}\psi^{j + k +m-1} - \psi^{j+k} \right\|\\
&\\
& \leq \left(\prod_{n=j}^{j+k+m}  2\left\|A^{(n)}\right\|\right)  \left\| \psi^{j + k +m} \right\|\epsilon_{j+k+m} \\
& \ \ \ \ \ \ \ \ \ \ \ \ \ \ \ \ + \prod_{n=j}^{j+k} \left\|A^{(n)}\right\| \left\| B^{(j+k+1)}\cdots B^{(j+k+m-1)}\psi^{j + k +m-1} - \psi^{j+k} \right\|\\
&\\
& \leq 4^{-j-k-m} + \prod_{n=j}^{j+k} \left\|A^{(n)}\right\| \left\| B^{(j+k+1)}\cdots B^{(j+k+m-1)}\psi^{j + k +m-1} - \psi^{j+k} \right\|\\
&\\
& \leq 4^{-j-k-m} + 4^{-j-k -m+1}+ \prod_{n=j}^{j+k} \left\|A^{(n)}\right\| \left\| B^{(j+k+1)}\cdots B^{(j+k+m-2)}\psi^{j + k +m-2} - \psi^{j+k} \right\|\\
&\\
& \leq \ \ \cdots \\
&\\
& \leq \sum_{n=j+k}^{j+k+m} 4^{-n} \ .
\end{split}
\end{equation*}
Letting $m \to \infty$ the above estimate shows that
$$
\left\| A^{(j)}\cdots A^{(j+k)}(T\psi)^{j+k} - \psi^{j-1} \right\| \leq 4^{-j-k+1} \ ,
$$
so when we subsequently let $k \to \infty$ we find that $(ST\psi)^{j-1} = \psi^{j-1}$, and hence that $ST\psi = \psi$. We leave the reader to use the same method to show that $T\circ S$ is the identity on $\varprojlim_j B^{(j)}$.
\end{proof}

 The set of arrows in a Bratteli diagram $\Br$ can be described by the multiplicity matrices $\Br^{(j)}, j = 1,2,3, \cdots$, where $\Br^{(j)}$ is the matrix over $\Br_j \times \Br_{j-1}$ defined such that 
$$
\Br^{(j)}_{v,w} = \# r^{-1}(v) \cap s^{-1}(w) \ .
$$
By using multiplicity matrices we can quickly introduce an operation on Bratteli diagrams called telescoping. For the present purposes we need the observation that generalized gauge actions behave nicely with respect to telescoping. Let $\Br$ be a Bratteli diagram and $F : \Br_{Ar} \to \mathbb R$ a potential on $\Br$. Let $0 = k_0 <  k_1 < k_2 < \cdots$ be a strictly increasing sequence of natural numbers. Let $\Br'$ be the Bratteli diagram with level sets
$$
\Br'_j \ = \ \Br_{k_j} 
$$
and multiplicity matrices
$$
{\Br'}^{(j)} \ = \ \Br^{(k_j)}\Br^{(k_j-1)}\Br^{(k_j -2)} \cdots \Br^{(k_{j-1} + 1)} \ .
$$
The arrows in $\Br'$ from ${\Br'}_{j-1}$ to ${\Br'}_{j}$ can then be identified with the set of paths in $\Br$ from $\Br_{k_{j-1}}$ to $\Br_{k_j}$ and we define a potential $F': \Br'_{Ar} \to \mathbb R$ such that
$$
F'(a) = F(\mu_a) \ ,
$$
where $\mu_a$ is the path in $\Br$ corresponding to $a \in \Br'_{Ar}$. We say that the pair $(\Br',F')$ is obtained from $(\Br,F)$ by telescoping to $k_1 < k_2  < \cdots$. It is then straightforward to prove the following

\begin{lemma}\label{13-05-19} Assume that $(\Br',F')$ is obtained from $(\Br,F)$ by telescoping. The flow $\alpha^{F'}$ on $\AF(\Br')$ is conjugate to the flow $\alpha^F$ on $\AF(\Br)$.
\end{lemma}

\begin{lemma}\label{12-08-20} Let $\Br$ be a Bratteli diagram. Let $U$ be a $\UHF$ algebra and $\{m_j\}$ a sequence of natural numbers. There is a Bratteli diagram $\Br'$ such that 
\begin{enumerate}
\item $\Br_V = \Br'_V$, i.e. $\Br'$ has the same set of vertexes as $\Br$, 
\item $\Br'^{(j)}_{v,w} \geq \Br^{(j)}_{v,w} \ + \ m_j$ for all $(v,w) \in \Br_j \times \Br_{j-1}$ and all $j =1,2,3, \cdots$,
 and
\item $\AF(\Br')$ is $*$-isomorphic to $U$.
\end{enumerate}
\end{lemma}
\begin{proof} Except for the second condition this follows from Lemma 5.1 in \cite{Th2}. We check here that the construction in \cite{Th2} can be arranged to obtain (2). For $j =1,2,3, \cdots$, let $C_j$ be a set consisting of one element $v_j$, and let $\Br''$ be the Bratteli diagram with level sets
$\Br''_{2j-1} = \Br_j, j =1,2,\cdots$, $\Br''_0 = \Br_0$ and $\Br''_{2j} = C_j, \ j = 1,2,3, \cdots$. To define the multiplicity matrices of $\Br''$ let $d_j \geq 2,\ j = 1,2,3,\cdots$, be a sequence of natural numbers such that $U$ is isomorphic to $\AF(\Br''')$ where $\Br'''$ is the Bratteli diagram with one vertex at each level and $d_j$ arrows from level $j-1$ to level $j$. Since $U$ is infinite dimensional, $\lim_{k \to \infty} d_1d_2d_3 \cdots d_k \ = \ \infty$ and we can therefore choose natural numbers $0 = k_0 < k_1 < k_2 < \cdots$ such that when we write
$$
\prod_{i=k_{j-1} +1}^{k_j} d_i  = S_j\left(\# \Br_j\right) + r_j \ ,
$$
where $S_j,r_j \in \mathbb N$, $r_j \leq \#\Br_j$, the lower bound
$$
S_j \ \geq \ \Br^{(j)}_{v,w} \ + \ m_j \ \ \ \ \ \ \forall (v,w) \in \Br_j \times \Br_{j-1} 
$$
holds. Choose an element $u_j \in \Br_j = \Br''_{2j-1}$ for all $j \geq 1$. For $j = 1,2,3, \cdots$ set
\begin{equation*}\label{26-09-18}
\Br''^{(2j-1)}_{v,v_{j-1}} =  \begin{cases} S_j+r_j , & \ v = u_{j} ,\\  S_j , & \ v \in \Br_j \backslash \{ u_{j}\} \  \end{cases}
\end{equation*}
and
$$
\Br''^{(2j)}_{v_j,v} = 1
$$
for all $v \in \Br''_{2j-1}$. The matrices $\{\Br''^{(j)}\}$ are the multiplicity matrices of $\Br''$. Let $\Br'$ be the Bratteli diagram obtained by removing from $\Br''$ the even level sets $\Br''_{2j} = C_j = \{v_j\}, \ j = 1,2,\cdots$, and telescoping as explained above; i.e. we telescope $\Br''$ to $1 < 3<5<7 < \cdots$. Then $\Br'_V = \Br_V$ and it is clear that $\Br'$ has the first two properties (1) and (2). To see that $\AF(\Br') \simeq U$, remove the odd level sets $\Br''_{2j-1} = \Br_j, \ j \geq 1$, in $\Br''$ and telescope to $2 < 4 < 6 < \cdots$. The result is the Bratteli diagram $\Br'''$ for $U$ telescoped to $k_1 <k_2 < \cdots$. Hence $\AF(\Br') \simeq \AF(\Br'') \simeq U$.
\end{proof}

In the following, by a generalized gauge action on an $\AF$-algebra $A$ we mean a flow on $A$ which is conjugate to the flow $\alpha^F$ arising from a potential $F : \Br \to \mathbb R$.

\begin{lemma}\label{19-11-18} Let $A_+$ and $A_-$ be $\AF$ algebras and $U$ a $\UHF$ algebra. There is a generalized gauge action $\alpha$ on $U$ such that $\GS(\alpha)$ is affinely homeomorphic to $S(A_+)$, $\CS(\alpha)$ is affinely homeomorphic to $S(A_-)$ and there is a unique $\beta$-KMS state for $\alpha$ for all $\beta \in \mathbb R$.
\end{lemma}
\begin{proof} Let $\Br(\pm)$ be Bratteli diagrams such that $\AF(\Br(\pm)) \simeq A_{\pm}$. Let $v^{\pm}_0$ be the top vertexes in $\Br(\pm)$. We define a Bratteli diagram $\Br$ such that 
$$
\Br_j = \Br(+)_{j-1} \sqcup \Br(-)_{j-1} \ , \ j \geq 1 ,
$$
and $\Br_0 = \{v_0\}$. There is one arrow in $\Br_{Ar}$ from $v_0$ to $v^+_0$ and one from $v_0$ to $v^-_0$, and no other arrows from $\Br_0$ to $\Br_1$. For $(v,w) \in \Br_j \times \Br_{j-1}, \ j \geq 2$, we set
$$
\Br^{(j)}_{v,w} = \begin{cases} \Br(+)^{(j-1)}_{v,w} & \ \text{when} \ (v,w) \in \Br(+)_{j-1} \times \Br(+)_{j-2} \ , \\
 \Br(-)^{(j-1)}_{v,w} & \ \text{when} \ (v,w) \in \Br(-)_{j-1} \times \Br(-)_{j-2} \ , \\
 0  & \ \text{otherwise} \ . \end{cases}
 $$
Then $\AF(\Br) \simeq A_+\oplus A_-$. Note that there is an obvious identification
$$
P(\Br) = P(\Br(+)) \sqcup P(\Br(-)) \ .
$$ 
By Lemma \ref{12-08-20} there is a Bratteli diagram $\Br'$ such that
\begin{itemize}
\item $\Br'_V = \Br_V$,
\item $\Br^{(j)}_{v,w} \ + \ 1  \ \leq  \  \Br'^{(j)}_{v,w}$ for all $(v,w) \in \Br_j \times \Br_{j-1}$ and all $j$, and
\item $\AF (\Br') \simeq U$. 
\end{itemize} 
Choose first $\delta_j \in \left] 0,\frac{1}{2}\right[$ such that
$$
\delta_j \frac{ (\# \Br_j)\prod_{k=1}^j \left(2 \left\|\Br'^{(j)}\right\| +1\right)}{\left(\Br'^{(j)}\Br'^{(j-1)} \cdots \Br'^{(1)}\right)_{w,v_0}} \ \leq \ 4^{-j} \ 
$$
for all $w \in \Br'_j = \Br_j$, and then $\epsilon_j > 0$ such that
$$
\left|e^{\beta \epsilon_j} -1\right| \leq \delta_j
$$
for all $\beta \in [-j,j]$.
Let $\mathcal A_{v,w}^{(j)}$ and ${\mathcal A'}_{v,w}^{(j)}$ denote the set of arrows from $w \in \Br_{j-1}$ to $v \in \Br_j$ in $\Br$ and $\Br'$ respectively. Since $\Br'^{(j)}_{v,w} \geq \Br^{(j)}_{v,w}$ we can assume that $\mathcal A_{v,w}^{(j)} \subseteq{\mathcal A'}_{v,w}^{(j)}$, and we define a potential $F : \Br'_{Ar} \to \mathbb R$ such that $F(a) = 0$ when $s(a) = v_0$ and
$$
F(a) = \begin{cases} -\epsilon_j,  \ & \ \text{when} \ a \in \mathcal A_{v,w}^{(j)} \ \text{and} \ s(a) \in \Br(+)_{j-1}  \\ 0, \ & \  \text{when} \ a \in {\mathcal A'}_{v,w}^{(j)} \backslash \mathcal A_{v,w}^{(j)} \\
 \epsilon_j,  & \ \text{when} \ a \in \mathcal A_{v,w}^{(j)} \ \text{and} \ s(a) \in \Br(-)_{j-1} . \end{cases}
$$ 
It follows from the definition of $F$ and the fact that ${\Br'}^{(j)}_{v,w} \geq 1$ for all $v,w$, that an infinite path $p \in P(\Br')$ is an $F$-geodesic iff $p \in P(\Br(+))$ and a $-F$-geodesic iff $p \in P(\Br(-))$. It follows therefore from Theorem \ref{09-05-19d} that $\GS(\alpha^F) \simeq S\left(\AF(\Br(+))\right) \simeq S(A_+)$ and $\CS(\alpha^F) = \GS(\alpha^{-F}) \simeq S\left(\AF(\Br(-))\right) \simeq S(A_-)$, where the symbol $\simeq$ denotes affine homeomorphism. 
Note next that
$$
\left|\Br'(\beta)^{(j)}_{v,w} - \Br'^{(j)}_{w,v}\right| \ \leq \ \left|e^{\beta \epsilon_j} - 1\right| \Br'^{(j)}_{w,v} \ \leq \ \delta_j  \Br'^{(j)}_{w,v}
$$
for all $\beta \in [-j,j]$ and all $v \in \Br'_{j-1}, \ w \in \Br'_j$.  It follows therefore from Lemma \ref{09-05-19g} that $\varprojlim_j \Br'(\beta)^{(j)} \simeq \varprojlim_j \left(\Br'^{(j)}\right)^T$ for all $\beta \in \mathbb R$, where $\left(\Br'^{(j)}\right)^T$ denotes the transpose of $\Br'^{(j)}$. It is well-known that $\varprojlim_j \left(\Br'^{(j)}\right)^T$ is affinely homeomorphic to the cone of positive trace functionals on $\AF(\Br')$.\footnote{This can also be deduced from the identification $\varprojlim (\Br'^{(j)})^T = \varprojlim\Br'(0)^{(j)}$ and Lemma \ref{10-05-19} since the $0$-KMS states are the trace states.}Since $\AF(\Br') \simeq U$ this is $\mathbb R^+$. Thus $\varprojlim_j \Br'(\beta)^{(j)}$ is a copy of $\mathbb R^+$ and $\alpha^F$ has a unique $\beta$-KMS state for every $\beta\in \mathbb R$ by Lemma \ref{10-05-19}.
\end{proof}

\begin{lemma}\label{19-11-18a} Let $\alpha^F$ be a generalized gauge action on $\AF(\Br)$ and let $U$ be a $\UHF$ algebra. There is a generalized gauge action $\alpha$ on $U$ such that $\KMS_\beta(\alpha^F)$ is strongly affinely isomorphic to $\KMS_{\beta}(\alpha)$ for all $\beta \neq 0$ and there is exactly one ground state and one ceiling state for $\alpha$.
\end{lemma}
\begin{proof} Write $U = U^+\otimes U^-$ where $U^{\pm}$ are $\UHF$ algebras. It follows from Theorem 5.5 and Lemma 4.3 in \cite{Th2} that there is a Bratteli diagram $\Br'$ and a potential $F' : \Br'_{Ar} \to \mathbb R$ such that
\begin{itemize}
\item $\KMS_{\beta}(\alpha^F)$ is strongly affinely isomorphic to $\KMS_\beta(\alpha^{F'})$ for all $\beta \neq 0$,
\item $\AF(\Br') \simeq U^+$.
\end{itemize}
Since $U^+$ and hence also $\AF(\Br')$ is simple, the Bratteli diagram $\Br'$ has the property that for each $j$ there is an $n \in \mathbb N$ such that
$$
\left({\Br'}^{(j+n)}{\Br'}^{(j+n-1)} \cdots {\Br'}^{(j+1)}\right)_{v,w} \ \geq \ 1
$$
for all $w\in \Br'_j, \ v \in \Br'_{j+n}$, cf. Corollary 3.5 of \cite{Br}. Since $\AF(\Br)$ is infinite dimensional, it follows that we can arrange, after a telescoping of $\Br'$ and an application of Lemma \ref{13-05-19}, that
\begin{itemize}
\item $\Br'^{(j)}_{v,w} \geq 2$ for all $(v,w) \in \Br'_j \times \Br'_{j-1}$ and all $j = 1,2,3, \cdots$ . 
\end{itemize}
Exchanging $\Br$ and $F$ with $\Br'$ and $F'$ we can therefore assume that
\begin{itemize}
\item[(a)] $\AF(\Br) \simeq U^+$ and
\item[(b)] $\Br^{(j)}_{v,w} \geq 2$ for all $(v,w) \in \Br_j \times \Br_{j-1}$ and all $j = 1,2,3, \cdots$ . 
\end{itemize}
Let $\left\{\epsilon_j\right\}_{j=1}^\infty$ be a sequence of positive numbers in the interval $\left]0,\frac{1}{2}\right[$ such that
\begin{equation*}\label{07-08-20} \epsilon_j \sqrt{\# \Br_j}  \left(\left(\Br(\beta)^{(1)} \cdots \Br(\beta)^{(j)}\right)_{v_0,w} \right)^{-1} \prod_{k=1}^j \left(2\left\|\Br(\beta)^{(k)}\right\| +1 \right) \ \leq \ 4^{-j}
\end{equation*}
 for all $w \in \Br_j$ and all $\in [-j,j]$. For each $j \geq 1$ we choose real numbers $m^{\pm}_j$ such that
\begin{equation}\label{13-08-20}
m^+_j \ < \ F(a) + F(b)-F(c) \  < \ m^-_j
\end{equation}
for all arrows $a,b,c \in \Br_{Ar}$ with $r(a) \in \Br_j,$ and $r(b),r(c) \in \Br_{j+1}$. Let $q = (q_i)_{i=1}^{\infty} \in P(\Br)$ be an infinite path in $\Br$. Let $\{d_j\}_{j=1}^{\infty}$ be a sequence of natural numbers such that the Bratteli diagram with one vertex at each level and one-by-one multiplicity matrices given by the sequence $d_1,d_2,d_3, \cdots$ is a diagram for $U^-$. We can then choose a sequence $1 = i_0 < i_1 < i_2 < i_3 < \cdots$ in $\mathbb N$ such that the numbers
$$
D_j = \prod_{k=i_{j-1}}^{i_j} d_k
$$
have the properties that $D_j \geq 2$ and  
\begin{equation}\label{23-11-18}
\left| D_j^{-1}\left( e^{-\beta m^-_j} + e^{-\beta m^+_j} - 2e^{-\beta F(q_j)}\right)\right| \leq \epsilon_j e^{- \beta F(q_j)} 
\end{equation}
for all $\beta \in [-j,j]$ and all $j \geq 1$. Let $\Br'$ be the Bratteli diagram with $\Br'_V = \Br_V$ and multiplicity matrices $\Br'^{(j)} = D_j\Br^{(j)}, \ j =1,2,3, \cdots$. Then
$$
\AF(\Br') \simeq \AF(\Br) \otimes U^- \simeq U^+\otimes U^- \simeq U \ .
$$ 
To define a potential on $\Br'$ let $\mathcal A^{(j)}_{v,w}$ be the set of arrows in $\Br$ from $w \in \Br_{j-1}$ to $v \in \Br_j$. We identify the set ${\mathcal A'}^{(j)}_{v,w}$ of arrows from $w \in \Br'_{j-1}$ to $v \in \Br'_j$ with the set $\{1,2, \cdots, D_{j}\} \times {\mathcal A}^{(j)}_{v,w}$. Set
$$
{F'}(i,a) = F(a)
$$
when $(i,a) \in \{1,2, \cdots, D_{j}\} \times {\mathcal A}^{(j)}_{v,w} \backslash \{(1,q_j), (2,q_j)\}$ and 
$$
{F'}(1,q_j) = m^+_j \ , \ {F'}(2,q_j) = m^-_j  \ .
$$
Then
$$
{\Br'}(\beta)^{(j)}_{v,w} = D_j\Br(\beta)^{(j)}_{v,w}   
$$
when $(v,w) \neq (s(q_j),r(q_j))$ and
\begin{equation*}
\begin{split}
&{\Br'}(\beta)^{(j)}_{v,w} \\ 
& =  e^{-\beta m^-_j}  +  e^{-\beta m^+_j} \ - \ 2 e^{-\beta F(q_j)}  \ + \ D_j \Br(\beta)^{(j)}_{v,w}   
\end{split}
\end{equation*}
when $(v,w) = (s(q_j),r(q_j))$. It follows from \eqref{23-11-18} that
$$
\left| D_j^{-1} {\Br'}(\beta)^{(j)}_{v,w} - \Br(\beta)^{(j)}_{v,w} \right|  \ \leq \ \epsilon_j \Br(\beta)^{(j)}_{v,w}
$$
for all $(v,w) \in \Br_{j-1}\times \Br_j$ when $\beta \in [-j,j]$ and then from Lemma \ref{22-11-18a} and Lemma \ref{09-05-19g} that
$$
\varprojlim_j \Br(\beta)^{(j)} \simeq \varprojlim_j D_j^{-1} \Br'(\beta)^{(j)} \simeq \varprojlim_j \Br'(\beta)^{(j)} 
$$
for all $\beta \in \mathbb R$. It follows then from Proposition 3.4 and Lemma 4.3 in \cite{Th2} that $\KMS_\beta(\alpha^F)$ is strongly affinely isomorphic to $\KMS_\beta(\alpha^{F'})$ for all $\beta \in \mathbb R$. Since $\AF(\Br') \simeq U$ there is a one-parameter group $\alpha$ on $U$ conjugate to $\alpha^{F'}$, and it remains only to show that $\GS(\alpha)$ and $\CS(\alpha)$ both only consist of one element, or equivalently that $\GS(\alpha^{F'})$ and $\CS(\alpha^{F'})$ only contain one element. To do this for $\GS(\alpha^{F'})$ it suffices to show that $\Geo^{F'}(\Br') = \{q'\}$ where $q' = \left\{(1,q_j)_{j=1}^{\infty}\right\}$ since this in combination with Theorem \ref{09-05-19d} implies that $\GS(\alpha^{F'}) \simeq S(\mathbb C)$. It is clear that $q' \in \Geo^{F'}(\Br')$ since $F'(1,q_j) \leq  F'(i,a)$ for all $(i,a) \in r^{-1}(\Br'_j)$ and all $j$. To prove that it is the only element of $\Geo^{F'}(\Br')$, let $p = (p_j)_{j=1}^{\infty} \in P(\Br')$ such that $p \neq q'$. Let $j$ be the least natural number for which $p_j \neq q'_j$. Then
$$
F\left(p|_{[1,j+1]}\right) - F(q'_{[1,j+1]}) = F(p_{j+1}) + F(p_j) - F(q'_{j+1}) - F(q'_j) > 0 \ 
$$
since $F\left(q'_j\right) < F(p_j)$ and $F(q'_{j+1}) \leq F(p_{j+1})$. This shows that $p \notin \Geo^{F'}(\Br')$ when $r(p_{j+1}) = r(q'_{j+1})$. Assume therefore that $r(p_{j+1}) \neq r(q'_{j+1})$. Choose an arrow $a \in \Br'_{Ar}$ such that $s(a) = r(q'_j)$, $r(a) = r(p_{j+1})$ and $a \neq (2,q_{j+1})$ which is possible thanks to (b) above. Then
$$
F(q'_{[1,j]}a) - F(p|_{[1,j+1]})  =  m^+_j + F(a) - F(p_j) - F(p_{j+1}) \  .
$$
Since $p_j \neq q'_j$ and $p_{j+1} \neq q'_{j+1}$ it follows from \eqref{13-08-20} that $F(q'_{[1,j]}a) - F(p|_{[1,j+1]}) < 0$, and hence the conclusion is again that $p \notin \Geo^{F'}(\Br')$. Hence $\GS(\alpha^{F'}) \simeq S(\mathbb C)$ as asserted. An identical argument works to show that $\Geo^{-F'}(\Br) =  \left\{(2,q_j)_{j=1}^{\infty}\right\}$, leading to the conclusion that also $\CS(\alpha)$ only consists of one state.

\end{proof}

To simplify notation in the following, and to emphasize the potentials, we let $\mathbb G_F$ denote the $C^*$-algebra $\AF(\Br^+)$ when $\Br^+$ is defined using the potential $F: \Br_{Ar} \to \mathbb R$. 

\begin{lemma}\label{19-11-18b} Let $\Br$ and $\Br'$ be Bratteli diagrams with potentials $F:\Br_{Ar} \to \mathbb R$ and $F' : \Br'_{Ar} \to \mathbb R$. Then $\GS (\alpha^F \otimes \alpha^{F'}) \simeq S(\mathbb G_F \otimes \mathbb G_{F'})$.
\end{lemma}
\begin{proof} $\AF(\Br) \otimes \AF(\Br')$ is $*$-isomorphic to $\AF(\Br \times \Br')$ where 
$$
\left(\Br\times \Br'\right)_j = \Br_j \times \Br'_j
$$ 
with multiplicity matrices given by $\left(\Br \times \Br'\right)^{(j)}_{(v,w), (v',w')} = \Br^{(j)}_{v,v'}\Br'^{(j)}_{w,w'}$. Under this $*$-isomorphism $\alpha^F \otimes \alpha^{F'}$ is conjugate to the generalized gauge action on $\AF(\Br \times \Br')$ defined by a potential $F''$ such 
$$F''(a,a')= F(a) + F'(a')$$ 
when $(a,a') \in \left(\Br \times \Br'\right)_{Ar} \subseteq \Br_{Ar} \times \Br'_{Ar}$. Then
$$
\left(\Br \times \Br'\right)^+ \ =  \ \Br^+ \times \Br'^+ \ ,
$$
implying that
$$
\mathbb G_{F''} \simeq \mathbb G_F \otimes \mathbb G_{F'} \ .
$$
It follows therefore from Theorem \ref{09-05-19d} that $\GS\left(\alpha^F \otimes \alpha^{F'}\right) \simeq S\left( \mathbb G_{F''} \right) \simeq S\left(\mathbb G_F \otimes \mathbb G_{F'}\right) $.
\end{proof}

The following is the main result of the paper. It is a restatement of Theorem \ref{MAIN0} from the introduction.

\begin{thm}\label{MAIN1} Let $F$ be a potential on the Bratteli diagram $\Br$ and let $A_+$ and $A_-$ be $\AF$ algebras. Let $U$ be a $\UHF$ algebra. There is a generalized gauge action $\alpha$ on $U$ such that $\KMS_{\beta}(\alpha)$ is strongly affinely isomorphic to $\KMS_\beta(\alpha^F)$ for all $\beta \neq 0$ while $\GS(\alpha)$ is affinely homeomorphic to $S(A_+)$ and $\CS(\alpha)$ is affinely homeomorphic to $S(A_-)$.
\end{thm}
\begin{proof} Write $U \simeq U_1 \otimes U_2$ where $U_i, i =1,2$ are UHF algebras. It follows from Lemma \ref{19-11-18} that there is a generalized gauge action $\alpha^1$ on $U_1$ such that $\GS(\alpha^1) \simeq S(A_+)$, $\CS(\alpha^1) \simeq S(A_-)$ and such that there is a unique $\beta$-KMS state $\omega_{\beta}$ for $\alpha^1$ for all $\beta$. Let $\Br^1$ be a Bratteli diagram and $F_1$ a potential on $\Br^1$ such that $\alpha^1$ is conjugate to $\alpha^{F_1}$. Then $S(\mathbb G_{F_1}) \simeq S(A_+)$ and $S\left(\mathbb G_{-F_1}\right) \simeq S(A_-)$ by Theorem \ref{09-05-19d}. It follows from Lemma \ref{19-11-18a} that there is a generalized gauge action $\alpha^2$ on $U_2$ such that $\KMS_\beta(\alpha^2)$ is strongly affinely homeomorphic to $\KMS_\beta(\alpha^F)$ for all $\beta \neq 0$ while $\GS(\alpha^2)$ and $\CS(\alpha^2)$ both only contain one state. Let $\Br^2$ be a Bratteli diagram and $F_2$ a potential on $\Br^2$ such that $\alpha^2$ is conjugate to $\alpha^{F_2}$. Note that $\mathbb G_{F_2} \simeq \mathbb G_{-F_2} \simeq \mathbb C$. Set $\alpha_t = \alpha^1_t \otimes \alpha^2_t$ for $t \in \mathbb R$ and note that $\alpha$ is conjugate to $\alpha^{F_1} \otimes \alpha^{F_2}$. We repeat then an argument from the proof of Lemma 5.4 in \cite{Th2} to show that $\KMS_\beta(\alpha) \simeq \KMS_\beta(\alpha^F)$: As in the proof of Lemma \ref{19-11-18b} $\alpha^{F_1} \otimes \alpha^{F_2}$ is conjugate to $\alpha^{F''}$ where the potential $F'': \Br^1 \times \Br^2 \to \mathbb R$ is defined such that
$$
F''(a,a')= F_1(a) + F_2(a') \ .
$$ 
Let $\left\{E^n_{\mu,\mu'}\right\}$ and $\left\{E'^n_{\nu,\nu'}\right\}$ be the matrix units from \eqref{22-09-18} in $\AF(\Br^1)$ and $\AF(\Br^2)$, respectively. Recall that the canonical diagonal $D$ of $\AF(\Br^2)$ is generated by the projections $E^n_{\mu,\mu}, \ \mu \in \mathcal P_n, \ n \in \mathbb N$, and that there is a conditional expectation $P : \AF(\Br^2) \to D$. When $\omega$ is a $\beta$-KMS state for $\alpha^{F''}$ we find that
\begin{equation*}
\begin{split}
&\omega\left( E^n_{\mu,\mu'} \otimes E'^n_{\nu,\nu'}\right) = \delta_{\mu,\mu'} \delta_{\nu,\nu'} \omega\left(E^n_{\mu, \mu} \otimes E'^n_{\nu,\nu}\right) = \omega \left( E^n_{\mu, \mu'} \otimes P(E'^n_{\nu, \nu'})\right) \ ,
\end{split}
\end{equation*}
showing that $\omega$ factorises through $\id_{\AF(\Br^1)} \otimes P$. Let $b \in \AF(\Br^2)$ be a positive element fixed by $\alpha^{F_2}$. Then $a \mapsto \omega(a  \otimes b)$ is a non-negative multiple of the unique $\beta$-KMS functional $\omega_\beta$ for $\alpha^{F_1}$ and hence  
$$
\omega( a \otimes b) = \lambda(b) \omega_\beta(a) \ \ \forall a \in \AF(\Br^1)
$$
for some $\lambda(b) \geq 0$. Since the diagonal $D \subseteq \AF(\Br^2)$ is in the fixed point algebra of $\alpha^{F_2}$ and $\omega$ factorises through $\id_{\AF(\Br^1)}\otimes P$ it follows that
$\omega = \omega_\beta \otimes \omega''$ for some $\omega'' \in \KMS_\beta(\alpha^{F_2})$. This shows that $\KMS_\beta(\alpha^{F''})$ is affinely homeomorphic to $\KMS_\beta(\alpha^{F_2})$ under the map $\omega \mapsto \omega_\beta \otimes \omega$ and hence $\KMS_\beta(\alpha) \simeq \KMS_\beta(\alpha^F)$. 

To complete the proof note that Lemma \ref{19-11-18b} implies that
$$
\GS(\alpha) \simeq S\left(\mathbb G_{F_1} \otimes \mathbb G_{F_2}\right) \simeq  S\left(\mathbb G_{F_1} \right) \simeq \GS(\alpha^1) \simeq S(A_+) \ ,
$$
and similarly,
$$
\CS(\alpha) \simeq S\left(\mathbb G_{-F_1} \otimes \mathbb G_{-F_2}\right) \simeq  S\left(\mathbb G_{-F_1} \right) \simeq \CS(\alpha^1) \simeq S(A_-) \ .
$$
\end{proof}

\section{On the $KMS_{\infty}$-states} Let $\alpha$ be a flow on the unital $C^*$-algebra $A$. Recall, \cite{CM}, that a state $\omega$ on $A$ is a $\KMS_{\infty}$-state when there is a sequence $\{\beta_k\}$ of real numbers such that $\lim_{k \to \infty} \beta_k = \infty$ and a sequence $\{\omega_k\}$ of states on $A$ such that $\omega_k$ is a $\beta_k$-KMS state for $\alpha$ and $\lim_{k \to \infty} \omega_{k} = \omega$ in the weak* topology.\footnote{There is an alternative definition which may be closer to what Connes and Marcolli had in mind. See \cite{CMN}. However, the definition we use here has been adopted in much of the subsequent work, including \cite{LR} and \cite{LLN}. It is not clear if the various definitions agree or not.} When $\alpha$ is approximately inner and $A$ has a trace state there are $\beta$-KMS states for all $\beta \in \mathbb R$ by a result of Powers and Sakai, \cite{PS}, and the compactness of $S(A)$ implies therefore that such flows also have $\KMS_\infty$-states. In particular, the flows we consider in this paper all have $\KMS_\infty$-states. 

Let $F :\Br_{Ar} \to \mathbb R$ be a potential on the Bratteli diagram $\Br$. A state $\omega$ of $\AF(\Br)$ is called a local $\KMS_\infty$-state for $\alpha^F$ when the restriction $\omega|_{\AF_j(\Br)}$ of $\omega$ to $\AF_j(\Br)$ is a $\KMS_\infty$-state for $\alpha^F|_{\AF_j(\Br)}$ for all $j$. A $\KMS_\infty$-state is also a local $\KMS_\infty$-state, but the converse is not generally true, cf. Example \ref{15-09-20}.

For the study of (local) $\KMS_\infty$-states we need some notation. For $j \geq 1, \ v \in \Br_j$, set
$$
\mathcal P^v_j = \left\{ \mu \in \mathcal P_j: \ r(\mu) = v \right\} \ ,
$$
and
$$
p^v_j = \sum_{\mu \in \mathcal P^v_j} E^j_{\mu,\mu} \ ,
$$
and set $p^{v_0}_1 = 1$. Let $j \geq 1$. Then $p^v_j$ is a central projection in $\AF_j(\Br)$ and $p^v_j \AF_j(\Br) \simeq M_{n}(\mathbb C)$ where $n = \# \mathcal P^v_j$. There is therefore a unique positive tracial functional $\Tr_v : \AF_j(\Br) \to \mathbb C$ such that 
$$
\Tr_v\left(p^w_j\right) = \begin{cases} \# \mathcal P^v_j \ , & \ w = v \\ 0 \ , & \ w \neq v \ . \end{cases} 
$$
The flow $\Ad e^{itH_j}$ on $\AF_j(\Br)$ leaves $p^v_j\AF_j(\Br)$ globally invariant and since $\beta$-KMS states are unique on matrix algebras it follows that a $\beta$-KMS state $\omega \in S\left(\AF_j(\Br)\right)$ for the restriction of $\alpha^F$ to $\AF_j(\Br)$ has the form
$$
\omega(\ \cdot \ ) \ = \ \sum_{v \in \Br_j} \omega(p^v_j) \frac{\Tr_v \left( e^{-\beta H_j} \ \cdot \ \right)}{\Tr_v\left(e^{-\beta H_j}\right)} \ .
$$
For $v \in \Br_j$, set $m_v = \min \left\{F(\mu) : \ \mu \in \mathcal P^v_j\right\}$,
$$
\mathcal M^v_j \ = \  \left\{ \mu \in \mathcal P_j^v: \ F(\mu) = m_v \right\} \ ,
$$
and
$$
A_v \ = \ \frac{1}{\#  \mathcal M^v_j} \sum_{\mu \in \mathcal M^v_j} E^j_{\mu , \mu} \ .
$$

\begin{lemma}\label{10-09-20c} A state $\omega$ on $\AF(\Br)$ is a local $\KMS_\infty$ state if and only 
\begin{equation}\label{10-09-20f}
\omega|_{\AF_j(\Br)}( \ \cdot \  ) \ = \ \sum_{v \in \Br_j} \omega\left(p^v_j\right) \Tr_v\left(A_v \ \cdot \ \right) 
\end{equation}
for all $j \geq 1$. 
\end{lemma}
\begin{proof} Assume that $\omega$ is local $\KMS_\infty$-state and let $\{\omega_k\}$ be a sequence of states on $\AF_j(\Br)$ and $\{\beta_k\}$ a sequence of real numbers such that $\omega_k$ is a $\beta_k$-KMS state for $\alpha^F|_{\AF_j(\Br)}$, $\lim_{k \to \infty} \omega_k = \omega|_{\AF_j(\Br)}$ and $\lim_{k \to \infty} \beta_k = \infty$. Then
\begin{align}
&\omega_k(\ \cdot \ ) \ = \ \sum_{v \in \Br_j} \omega_k(p^v_j) \frac{\Tr_v \left( e^{-\beta_k H_j} \ \cdot \ \right)}{\Tr_v\left(e^{-\beta_k H_j}\right)}\label{10-09-20e} \\
&=  \sum_{v \in \Br_j} \omega_k(p^v_j) \frac{\Tr_v \left( e^{-\beta_k H_j}p^v_j \ \cdot \ \right)}{\Tr_v\left(e^{-\beta_k H_j}\right)} \ . \nonumber
\end{align} 
Fix $v \in \Br_j$ and let $I = \left\{ F(\mu) : \ \mu \in \mathcal P^v_j\right\}$. For $t \in I$, set
$$
p_t \ = \sum_{\mu \in \mathcal P^v_j, F(\mu) = t} \  \ E^j_{\mu,\mu} \ .
$$
Then
$$
e^{-\beta_k H_j}p^v_j  \ = \ \sum_{t \in I} e^{-\beta_k t}p_t \ 
$$
and
$$
 \frac{e^{-\beta_k H_j}p^v_j}{\Tr_v\left(e^{-\beta_k H_j}\right)} \ = \ \frac{\sum_{\mu \in \mathcal M^v_j} E^j_{\mu, \mu}  \ + \ \sum_{t \in I \backslash \{m_v\}}  e^{-\beta_k (t-m_v)}p_t  }{ \# \mathcal M^v_j \ + \ \sum_{t \in I \backslash \{m_v\}} \#\{\mu \in \mathcal P^v_j, \ F(\mu) = t\} e^{-\beta_k (t-m_v)}} \ ,
 $$
from which it follows that 
$$
\lim_{k \to \infty}\frac{e^{-\beta_k H_j}p^v_j}{\Tr_v\left(e^{-\beta_k H_j}\right)} \ = \ A_v \ .
$$  
Thus \eqref{10-09-20f} follows from \eqref{10-09-20e} by letting $k$ go to infinity. For the converse it suffices to show that $\sum_{v \in \Br_j} t^j_v \Tr_v(A_v \ \cdot \ )$ is a $\KMS_\infty$-state for the flow $\Ad e^{itH_j}$ on $\AF_j(\Br)$ when $t^j_v \geq 0$ and $\sum_{v \in \Br_j}t^j_v = 1$. This follows from the previous arguments since they show that
$$
\lim_{\beta \to \infty} \sum_{v \in \Br_j} t^j_v \frac{\Tr_v \left( e^{-\beta H_j} \ \cdot \ \right)}{\Tr_v\left(e^{-\beta H_j}\right)} \ = \ \sum_{v \in \Br_j} t^j_v \Tr_v(A_v \ \cdot \ ) \ .
$$
\end{proof}

\begin{lemma}\label{15-09-20a} A local $\KMS_\infty$-state for $\alpha^F$ is a ground state for $\alpha^F$.
\end{lemma}
\begin{proof} When $\mu \in \mathcal P_j \backslash \mathcal G_j$ there is an $n_\mu \geq j$ such that $\nu \notin \mathcal M^{r(\nu)}_{|\nu|}$ when $|\nu| \geq n_\mu$ and $\nu[1,j] = \mu$. Since $1 - Q_j = \sum_{\mu \in \mathcal P_j \backslash \mathcal G_j} \ E^j_{\mu, \mu}$ it follows that there is a $k \geq j$ such that $A_v(1-Q_j) = 0$ for all $v \in \Br_k$. By Lemma \ref{10-09-20c} this implies that
$$
\omega(1-Q_j) \ = \  \sum_{v \in \Br_k} \omega\left(p^v_k\right) \Tr_v\left(A_v(1-Q_j) \right) \ = \ 0 \ .
$$ 
Hence $\omega(Q_j) = 1$ for all $j$ and $\omega$ is a ground state by Proposition \ref{09-05-19b}.
\end{proof}

\begin{lemma}\label{10-09-20i} Let $\tau \in T(\AF(\Br^+))$ be a trace state on $\AF(\Br^+)$. Then $\tau \circ Q_F$ is a local $\KMS_\infty$-state.
\end{lemma}

\begin{proof} By Lemma \ref{10-09-20c} we must show that
$$
\tau \circ Q_F|_{\AF_j(\Br)}(\ \cdot \ ) \ = \ \sum_{v \in \Br_j} \tau\circ Q_F(p^v_j) \Tr_v(A_v \ \cdot \ ) \ .
$$
Both sides of the equation are states that are $e^{it H_j}$-invariant and hence they both factor through the conditional expectation $R : \AF_j(\Br) \to \AF_j(\Br)^{\alpha^F}$. It suffices therefore to show that they agree on $\AF_j(\Br)^{\alpha^F}$ where both are trace states. It suffices therefore to show that they agree on the minimal central projections in $\AF_j(\Br)^{\alpha^F}$. Fix $v \in \Br_j$ and note that the central projection in $p^v_j\AF_j(\Br)^{\alpha^F}$ are the projections $p_t, t \in I$, from the proof of Lemma \ref{10-09-20c}. We must show that the two states agree on $p_t$ for each $t \in I$. Since $\tau \circ Q_F(\ \cdot \ ) = \tau \circ Q_F(Q_j \ \cdot \ Q_j)$ it follows that $\tau \circ Q_F$ annihilates $p_t$ when $t \neq m_v$. By definition of $A_v$ the same is true for the other state in question and hence it suffices to show that they agree on $p_{m_v}$. For this note that $\tau\circ Q_F(p^v_j) \Tr_v(A_v p_{m_v}) = \tau\circ Q_F(p^v_j) \Tr_v(A_v) = \tau\circ Q_F(p^v_j)$ since $A_vp_{m_v} = A_v$ and $\Tr_v(A_v) = 1$, while $\tau \circ Q_F(p_{m_v}) = \tau \circ Q_F(p^v_j)$ because $Q_jp^v_jQ_j = Q_jp_{m_v}Q_j$. 
\end{proof}

\begin{thm}\label{10-09-20h} The map 
$$
T(\AF(\Br^+)) \ni \tau \ \mapsto  \ \tau \circ Q_F
$$
is an affine homeomorphism of $T(\AF(\Br^+))$ onto the set of local $\KMS_\infty$-states for $\alpha^F$.
\end{thm}
\begin{proof} The map takes values in the set of local $\KMS_\infty$-states by Lemma \ref{10-09-20i} and it is clearly affine, continuous and injective. Let $\omega$ be a local $\KMS_\infty$-state. Then $\omega$ is a ground state by Lemma \ref{15-09-20a} and hence $\omega = \tau \circ Q_F$ for some state $\tau$ on $\AF(\Br^+)$ by Proposition \ref{09-05-19b}. Since $A_v$ is a central element in $\AF_j(\Br)^{\alpha^F}$, the formula for $\omega|_{\AF_j(\Br)}$ in Lemma \ref{10-09-20c} implies that $\omega$ is a trace state on $\AF_j(\Br)^{\alpha^F}$. This is true for all $j$ and hence $\omega$ is a trace on $\AF(\Br)^{\alpha^F}$. Since $\AF(\Br^+) = Q_F\left(\AF(\Br)^{\alpha^F}\right)$ this implies that $\tau \in T(\AF(\Br^+))$.
\end{proof}

The next step will to show that in some cases all local $\KMS_\infty$-states are in fact $KMS_\infty$-states, while in others this is not so. When $w \in \Br_j$ and $v \in \Br_{j-1}$, set
$$
\mathcal P^w_j(v) = \left\{ \nu \in \mathcal P^w_j : \ v \in \nu \right\} \ .
$$
We define a matrix $\underline{\Br(\beta)}^{(j)}$ over $\Br_{j-1} \times \Br_j$ such that
$$
\underline{\Br(\beta)}^{(j)}_{v,w} \ = \ \frac{\sum_{\mu \in \mathcal P_{j}^w(v)} e^{-\beta F(\mu)}}{\sum_{\nu \in \mathcal P_j^w} e^{-\beta F(\nu)} } \ .
$$
In particular, $\underline{\Br(\beta)}^{(1)}_{v_0,v} = 1$ for all $v \in \Br_1$. Note that $\underline{\Br(\beta)}^{(j)}$ is left stochastic, i.e.
$$
\sum_{v \in \Br_{j-1}} \underline{\Br(\beta)}^{(j)}_{v,w} \ = \ 1 \ , \ \forall w \in \Br_j \ .
$$
Consider the simplexes
$$
\Delta_{\Br_j} \ = \ \left\{ (t_v)_{v \in \Br_j} \in [0,1]^{\Br_j} : \ \sum_{v \in \Br_j} t_v = 1 \right\} \ 
$$
for $j =0,1,2, \cdots $ and note that $\underline{\Br(\beta)}^{(j)}\Delta_{\Br_j} \subseteq \Delta_{\Br_{j-1}}$ since $\underline{\Br(\beta)}^{(j)}$ is left stochastic. The projective limit 
$$
\varprojlim_j \left(\Delta_{\Br_j}, \underline{\Br( \beta)}^{(j)}\right) \ = \ \left\{ (\psi^j)_{j =0}^\infty  \in \prod_{j=0}^{\infty} \Delta_{\Br_j} : \ \psi^j = \underline{\Br(\beta)}^{(j+1)} \psi^{j+1} \ \ \forall  j \right\} 
$$ 
is a Choquet simplex affinely homeorphic to the simplex $\KMS_\beta(\alpha^F)$ of $\beta$-KMS states for $\alpha^F$.
We shall only need the following half of this assertion.

\begin{lemma}\label{11-09-20} Let $\left(\left(t^v_j\right)_{v \in \Br_j}\right)_{j=1}^\infty \in \varprojlim_j \left(\Delta_{\Br_j}, \underline{\Br( \beta)}^{(j)}\right)$. There is a $\beta-\KMS$ state $\omega^\beta$ for $\alpha^F$ such that
$$
\omega^\beta\left(p^v_j\right) = t^v_j
$$
for all $v,j$.
\end{lemma}
\begin{proof} Define a state $\omega_j$ on $\AF_j(\Br)$ such that
$$
\omega_j( \ \cdot \ ) \ = \ \sum_{v \in \Br_j} t^v_j  \frac{\Tr_v \left( e^{-\beta H_j} \ \cdot \ \right)}{\Tr_v\left(e^{-\beta H_j}\right)} \ .
$$
To check that $\omega_j =\omega_{j+1}|_{\AF_j(\Br)}$, note that both sides are $\beta$-KMS states for $\alpha^F|_{\AF_j(\Br)} = \Ad e^{it H_j}$ and hence it suffices to show that they agree on $p^v_j$. We therefore calculate
\begin{align*}
& \omega_{j+1}\left( p^v_j \right) = \omega_{j+1}\left( \sum_{\nu \in \mathcal P^v_j} \sum_{ a \in s^{-1}(v)} E^{j+1}_{\nu a, \ \nu a}\right) \\
&= \sum_{w \in \Br_{j+1}} t^w_{j+1} \frac{\Tr_w \left( e^{-\beta H_{j+1}}\sum_{\nu \in \mathcal P^v_j} \sum_{ a \in s^{-1}(v) \cap r^{-1}(w)} E^{j+1}_{\nu a, \ \nu a} \right)}{\Tr_w\left(e^{-\beta H_{j+1}}\right)}\\
&  =  \sum_{w \in \Br_{j+1}} t^w_{j+1} \frac{\sum_{\mu' \in \mathcal P^w_{j+1}(v)} e^{-\beta F(\mu')}}{\sum_{\mu \in \mathcal P^w_{j+1}} e^{-\beta F(\mu)}}\\
& = \sum_{w \in \Br_{j+1}} t^w_{j+1} \underline{\Br(\beta)}^{(j+1)}_{v,w} \ = \ t^v_j  \ = \  \omega(p^v_j) \ .
\end{align*}
It follows that there is a state $\omega^\beta$ on $\AF(\Br)$ such that $\omega^\beta|_{\AF_j(\Br)} = \omega_j$ for all $j$. By construction $\omega^\beta|_{\AF_j(\Br)}$ is a $\beta$-KMS state for $\alpha^F|_{\AF(\Br)}$ for all $j$, and hence $\omega^\beta$ is a $\beta$-KMS state for $\alpha^F$ with the required property.
\end{proof}

Define ${\underline{\Br}}^{(j)}$ over $\Br_{j-1} \times \Br_j, \ j \geq 1$ such that
$$
{\underline{\Br}}^{(j)}_{v,w} \ =  \ \frac{\#  \left\{ \mu \in \mathcal P_j^w(v): \ F(\mu) = m_w \right\}}{\# \left\{ \mu \in \mathcal P_j^w: \ F(\mu) = m_w \right\}} \ .
$$ 
\begin{lemma}\label{04-09-20} $\lim_{\beta \to \infty} \underline{\Br(\beta)}^{(j)} \ = \ {\underline{\Br}}^{(j)}$.
\end{lemma}
\begin{proof} Set
$$
X(\beta)_{v,w} \ = \ \sum_{\mu \in \mathcal P^w_j(v) , \ F(\mu) > m_w} \ \ e^{-\beta (F(\mu) - m_w)}  
$$
and
$$
Y(\beta)_w  = \  \sum_{\mu \in \mathcal P^w_j , \ F(\mu) > m_w} \ \ e^{-\beta (F(\mu) - m_w)}  \ .
$$ 
Then
\begin{align*}
&\underline{\Br(\beta)}^{(j)}_{v,w}  = \frac{\#  \left\{ \mu \in \mathcal P_j^w(v): \ F(\mu) = m_w \right\} \ + \ X(\beta)_{v,w}}{\# \left\{ \mu \in \mathcal P_j^w: \ F(\mu) = m_w \right\} \ + \ Y(\beta)_w}
\end{align*}
and the conclusion follows because $\lim_{\beta \to \infty} X(\beta)_{v,w} = \lim_{\beta \to \infty} Y(\beta)_w  = 0$ for all $v,w$.
\end{proof}

Now we aim to show that when the convergence in Lemma \ref{04-09-20} is sufficiently uniform, all local $\KMS_\infty$-states are genuine $\KMS_\infty$-states. For this we shall use the $l^1$-norm on vectors in $\mathbb R^n$ and the corresponding operator norm on matrices. Both will be denoted by $\| \ \cdot \ \|_1$. The advantage of this norm is that left stochastic matrices are then weak contractions.

\begin{lemma}\label{03-09-20} Assume that $\lim_{\beta \to \infty } \sum_{j=1}^{\infty} \left\|\underline{\Br(\beta)}^{(j)} - {\underline{\Br}}^{(j)} \right\|_1 \ = \ 0$. For every $\psi \in \varprojlim_j \left(\Delta_{\Br_j},{\underline{\Br}}^{(j)}\right)$ there is a collection 
$$
\phi^\beta \in \varprojlim_j \left( \Delta_{\Br_j}, \underline{\Br(\beta)^j} \right) \ , \beta  \in \mathbb R \ ,
$$
such that $\lim_{\beta \to \infty} \phi^\beta = \psi$ in 
$\prod_{n=0}^{\infty} \Delta_{\Br_n}$.
\end{lemma}
\begin{proof} Let $\psi \in \varprojlim_j \left(\Delta_{\Br_j},{\underline{\Br}}^{(j)}\right)$. For $k \in \mathbb N$, $\beta \in \mathbb R$, define $\phi^{\beta}(k) \in \prod_{n=0}^{\infty} \Delta_{\Br_n}$
such that
$$  
\phi^{\beta}(k)_j = \begin{cases} \Br(\beta)^{(j+1)}\Br(\beta)^{(j+2)} \cdots \Br(\beta)^{(k)}\psi_k , & \ j < k \\ \psi_j , & \ j \geq k \ . \end{cases} 
$$
Let $\phi^\beta$ be a condensation point of $\left\{\phi^\beta(k)\right\}_{k=1}^\infty$ in $\prod_{n=0}^{\infty} \Delta_{\Br_n}$. Then $\phi^\beta \in \varprojlim_j \left( \Delta_{\Br_j}, \underline{\Br(\beta)^j} \right)$. To see that $\lim_{\beta \to \infty} \phi^\beta = \psi$, let $j \in \mathbb N$ and $\epsilon >0$ be given. By assumption we can choose $\beta_0$ such that $\sum_{n=1}^{\infty} \left\|\underline{\Br(\beta)}^{(n)} - {\underline{\Br}}^{(n)} \right\|_1 \leq \epsilon$ when $\beta \geq \beta_0$. Let $\beta \geq \beta_0$. By construction of $\phi^\beta$ there is a $k > j$ such that
\begin{align*} 
&\left\|\psi_j - \phi^{\beta}_j\right\|_1  \ \leq \ \left\|\psi_j - \underline{\Br(\beta)}^{(j+1)}\underline{\Br(\beta)}^{(j+2)} \cdots \underline{\Br(\beta)}^{(k)}\psi_k \right\|_1 + \epsilon\\
&\\
& = \left\| {\underline{\Br}}^{(j+1)}\cdots {\underline{\Br}}^{(k)}\psi_k  - \underline{\Br(\beta)}^{(j+1)} \cdots \underline{\Br(\beta)}^{(k)}\psi_k \right\|_1 + \epsilon\\
& \\
& \leq \left\|\underline{\Br(\beta)}^{(k)} - {\underline{\Br}}^{(k)}\right\|_1\\
& + \left\| {\underline{\Br}}^{(j+1)}\cdots {\underline{\Br}}^{(k)}\psi_k  - \underline{\Br(\beta)}^{(j+1)}\cdots \underline{\Br(\beta)}^{(k-1)}{\underline{\Br}}^{(k)}\psi_k \right\|_1 + \epsilon \ \\
& \\
&  \leq \left\|\underline{\Br(\beta)}^{(k)}  - {\underline{\Br}}^{(k)}\right\|_1 +\left\|\underline{\Br(\beta)}^{(k-1)}  - {\underline{\Br}}^{(k-1)}\right\|_1 \\
& + \left\| {\underline{\Br}}^{(j+1)}\cdots {\underline{\Br}}^{(k)}\psi_k  - \underline{\Br(\beta)}^{(j+1)} \cdots \underline{\Br(\beta)}^{(k-2)}{\underline{\Br}}^{(k-1)}{\underline{\Br}}^{(k)}\psi_k \right\|_1 + \epsilon \\
& \\
& \leq \ \ \cdots
& \\
& \leq \sum_{n=j+1}^k \left\|\underline{\Br(\beta)}^{(n)} - {\underline{\Br}}^{(n)} \right\|_1 + \epsilon \ \leq \ 2 \epsilon \ .
\end{align*}
\end{proof}

\begin{thm}\label{10-09-20j} Assume that
$\lim_{\beta \to \infty } \sum_{j=1}^{\infty} \left\|\underline{\Br(\beta)}^{(j)} - {\underline{\Br}}^{(j)} \right\|_1 \ = \ 0$.
Then all local $\KMS_\infty$-states for $\alpha^F$ are $\KMS_\infty$-states for $\alpha^F$, and hence the set of $\KMS_\infty$-states is convex and affinely homeomorphic to the tracial state space $T(\AF(\Br^+))$ of $\AF(\Br^+)$ via the map of Theorem \ref{10-09-20h}.
\end{thm}
\begin{proof} Let $\omega$ be local $\KMS_\infty$-state. We claim that
\begin{equation}\label{10-9-20k}
\left( \left( \omega(p^v_j)\right)_{v \in \Br_j}\right)_{j=0}^\infty \ \in \ \varprojlim \left( \Delta_{\Br_j}, {\underline{\Br}}^{(j)}\right) \ .
\end{equation}

The verification is a direct check based on \eqref{10-09-20f}:
\begin{align*}
& \omega(p^v_j) = \omega\left( \sum_{w \in \Br_{j+1}} \sum_{\mu \in \mathcal P^v_j} \sum_{a \in s^{-1}(v) \cap r^{-1}(w)} E^{j+1}_{\mu a, \mu a}\right) \\
& =  \sum_{w \in \Br_{j+1}}  \omega\left(p^w_{j+1}\right) \sum_{\mu \in \mathcal P^v_j} \sum_{a \in s^{-1}(v) \cap r^{-1}(w)} \Tr_w\left( A_w E^{j+1}_{\mu a, \mu a} \right) \\
& = \sum_{w \in \Br_{j+1}}  \omega\left(p^w_{j+1}\right) \frac{\# \mathcal M^w_{j+1} \cap \mathcal P^w_{j+1}(v)}{\# \mathcal M^w_{j+1}} \\
& = \sum_{w \in \Br_{j+1}} \underline{\Br}^{(j+1)}_{v,w} \omega\left(p^w_{j+1}\right) . 
\end{align*}

By combining \eqref{10-9-20k} with Lemma \ref{03-09-20} and Lemma \ref{11-09-20} we deduce the existence of a $\beta$-KMS state $\omega^\beta$ for all $\beta\in \mathbb R$ such that $\lim_{\beta \to \infty} \omega^\beta(p^v_j) = \omega(p^v_j)$ for all $v,j$. It follows that $\lim_{\beta \to \infty} \omega^\beta = \omega$ in the weak* topology.
\end{proof}

\begin{example}\label{10-09-20l} To see Theorem \ref{10-09-20j} in action, consider the Bratteli diagram $\Br$ with potential $F$ given by the following diagram.
\begin{equation*}\label{11-12-18a*}
\begin{xymatrix}{ &\Br & \ar[ld]_-{0} v_0 \ar[rd]^-{0} \\
 & \ar[d]_-{0} \ar[rrd]^(0.3){2} &  & \ar[d]^-{0} \ar[lld]_(0.3){2} \\  
 & \ar[d]_-{0}  \ar[rrd]^(0.3){3} &  & \ar[d]^-{0} \ar[lld]_(0.3){3} \\  
 & \ar[d]_-{0}  \ar[rrd]^(0.3){4} &  & \ar[d]^-{0} \ar[lld]_(0.3){4} \\  
 & \ar[d]_-{0}  \ar[rrd]^(0.3){5} &  & \ar[d]^-{0} \ar[lld]_(0.3){5} \\  
 & \ar[d]_-{0}  \ar[rrd]^(0.3){6} &  & \ar[d]^-{0} \ar[lld]_(0.3){6} \\  
 & \ar[d]_-{0} \ar[rrd]^(0.3){7} &  & \ar[d]^-{0} \ar[lld]_(0.3){7} \\  
 &\vdots &\vdots  & \vdots }
 \end{xymatrix}
 \end{equation*}
 Then
 $$
 \underline{\Br(\beta)}^{(j)} \ = \ \left( \begin{matrix} \frac{1}{1+e^{-\beta j}} & \frac{e^{-\beta j}}{1+e^{-\beta j}}  \\ \frac{e^{-\beta j}}{1+e^{-\beta j}}  &  \frac{1}{1+e^{-\beta j}} \end{matrix} \right) 
 $$
 and $\underline{\Br}^{(j)} = \left(\begin{smallmatrix} 1 & 0 \\ 0 & 1 \end{smallmatrix}\right)$ when $j \geq 2$. Hence
 $$
 \left\| \underline{\Br(\beta)}^{(j)} - \underline{\Br}^{(j)} \right\|_1 \ \leq \  2e^{-\beta j} 
 $$
when $j \geq 2$, while $\left\| \underline{\Br(\beta)}^{(1)} - \underline{\Br}^{(1)} \right\|_1 =0$. It follows that
$$
\sum_{j=1}^{\infty} \left\|\underline{\Br(\beta)}^{(j)} - {\underline{\Br}}^{(j)} \right\|_1 \ \leq \ \frac{2 e^{-\beta}}{1-e^{-\beta}}
$$ 
when $\beta >0$ and hence Theorem \ref{10-09-20j} applies. We conclude that all local $\KMS_\infty$-states for $\alpha^F$ are $\KMS_\infty$-states. Since $\AF(\Br^+) \simeq \mathbb C^2$ we can combine with Theorem \ref{09-05-19d} to conclude that all ground states are $\KMS_\infty$-states, and that the set of ground states is affinely homeomorphic to the interval $[0,1]$; the state space of $\mathbb C^2$.
 
\end{example}

We show next by example that in general not all local $\KMS_\infty$-states are $\KMS_\infty$-states.

\begin{example}\label{15-09-20} Consider the flow $\alpha^{F_2}$ defined by the Bratteli diagram and potential described by the labelled graph $\Br_2$ in Example \ref{09-05-19e}. For $\beta \in \mathbb R$ consider the matrix
$$
A(\beta) = \left( \begin{matrix} 1 & e^{-\beta} \\ e^{-\beta} & 1 \end{matrix} \right) \ .
$$
Then $\Br_2(\beta)^{(1)} = \left( \begin{matrix} 1 & 1 \end{matrix} \right)$ and $\Br_2(\beta)^{(j)} = A(\beta)$ when $j \geq 2$. We claim that there is only one element $(\psi^j)_{j=0}^{\infty} \in \varprojlim_j \Br_2(\beta)^{(j)}$ with $\psi^0 = 1$. To see this let $(\psi^j)_{j=0}^{\infty}$ be such an element. The largest eigenvalue of $A(\beta)$ is $1+e^{-\beta}$ and the orthogonal projection onto the corresponding eigenspace is given by the matrix
$$
 \frac{1}{2}\left( \begin{matrix} 1 & 1 \\ 1 & 1 \end{matrix} \right) \ .
 $$
 It follows therefore from the Perron-Frobenius theorem that 
$$
\lim_{j \to \infty} \left( 1+e^{-\beta}\right)^{-j}A(\beta)^j \ = \ \frac{1}{2}\left( \begin{matrix} 1 & 1 \\ 1 & 1 \end{matrix} \right) \ .
$$
Let $\mathbb P : \left\{ (x,y) : \ x \geq 0,y \geq 0\right\} \backslash \{(0,0)\} \to \mathbb R^2$ be the function 
$$
\mathbb P(x,y) = \left(\frac{x}{x+y},\frac{y}{x+y}\right) \ .
$$
 By using that $\mathbb Ptz = \mathbb Pz$ when $t > 0$  we find that
\begin{align*}
&\psi^n = A(\beta)^j\psi^{j+n} = \left(\sum_{v \in \Br_n}\psi^n_v\right) \mathbb P\left(A(\beta)^j\psi^{j+n}\right)  \\
&= \left(\sum_{v \in \Br_n}\psi^n_v\right) \mathbb P\left(\left( 1+e^{-\beta}\right)^{-j}A(\beta)^j \frac{\psi^{j+n}}{\left\|\psi^{j+n}\right\|}\right)\\
&   =\lim_{j \to \infty}\left(\sum_{v \in \Br_n}\psi^n_v\right) \mathbb P\left(\left( 1+e^{-\beta}\right)^{-j}A(\beta)^j \frac{\psi^{j+n}}{\left\|\psi^{j+n}\right\|}\right)\\
&=  \lim_{j \to \infty}\left(\sum_{v \in \Br_n}\psi^n_v\right) \mathbb P\left(\frac{1}{2}\left( \begin{matrix} 1 & 1 \\ 1 & 1 \end{matrix} \right) \frac{\psi^{j+n}}{\left\|\psi^{j+n}\right\|}\right) = \left( \frac{\sum_{v \in \Br_n}\psi^n_v}{2}, \frac{\sum_{v \in \Br_n}\psi^n_v}{2}\right) \ 
\end{align*}
when $n \geq 1$. Thus the two coordinates of $\psi^n$ are the same when $n \geq 1$. It follows that $\psi$ is unique, and by Lemma \ref{10-05-19} there is therefore a unique $\beta$-KMS state $\omega^\beta$ for all $\beta \in \mathbb R$. Furthermore, from the description of the homeomorphism in Lemma \ref{10-05-19} it follows that $\omega^\beta$ takes the same value $\frac{1}{2}$ on the two projections $p^v_j, v \in \Br_j$, and the same must therefore be true for a $\KMS_\infty$ state, which is therefore unique. As pointed out in Example \ref{09-05-19e}, $\AF(\Br_2^+) \simeq \mathbb C^2$ and the set of ground states is affinely homeomorphic to $[0,1]$. The unique $\KMS_\infty$-state is the average of the two extremal ground states. By combining Theorem \ref{09-05-19d} and Theorem \ref{10-09-20h} it follows that all ground states are local $\KMS_\infty$-states because all states on $\AF(\Br_2^+)$ are trace states.

\end{example}



\begin{thebibliography}{WWWW} 


\bibitem[AM]{AM} H. Araki and T. Matsui, {\em Ground states of the XY-model}, Comm. Math. Phys. {\bf 101} (1985), 215-245. 


\bibitem[Br]{Br} O. Bratteli, {\em Inductive limits of finite dimensional $C^*$-algebras}, Trans. Amer. Math. Soc. {\bf 171} (1972), 195-234.

\bibitem[BR]{BR} O. Bratteli and D.W. Robinson, {\em Operator Algebras
    and Quantum Statistical Mechanics I + II}, Texts and Monographs in
  Physics, Springer Verlag, New York, Heidelberg, Berlin, 1979 and 1981. 
  
\bibitem[CM]{CM} A. Connes and M. Marcolli, {\em From physics to number theory via noncommutative geometry}, Frontiers in number theory, physics, and geometry. I, 269--347, Springer, Berlin, 2006.
 
 
 \bibitem[CMN]{CMN} A. Connes, M. Marcolli and N. Ramachandran, {\em KMS states and complex multiplication}, Selecta Math. (N.S.) {\bf 11} (2005), no. 3-4, 325–-347. 



  
   
\bibitem[KPRR]{KPRR} A. Kumjian, D. Pask, I. Raeburn and J. Renault, {\em Graphs, Groupoids, and Cuntz-Krieger algebras}, J. Func. Analysis {\bf 144} (1997), 505-541.     

\bibitem[LR]{LR} M. Laca and I. Raeburn, {\em Phase transition on the Toeplitz algebra of the affine semigroup over the natural numbers},
Adv. Math. {\bf 225} (2010), 643--688.





\bibitem[LLN]{LLN} N. Larsen, M. Laca and S. Neshveyev, {\em Ground states of groupoid $C^*$-algebras, phase transitions and arithmetic subalgebras for Hecke algebras},  J. Geom. Phys. {\bf 136} (2019), 268--283.



\bibitem[PS]{PS} R.T. Powers and S. Sakai, {\em Existence of ground
    states and KMS states for approximately inner dynamics},
  Comm. Math. Phys. {\bf 39} (1975), 273-288. 

\bibitem[Ru]{Ru} D. Ruelle, {\em Some remarks on the ground state of infinite systems in statistical mechanics}, Comm. Math. Phys. {\bf 11} (1969), 339-345. 



\bibitem[Th1]{Th1} K. Thomsen, {\em KMS weights on graph
    $C^*$-algebras II, Factor types and ground states}, arXiv:1412.6762



\bibitem[Th2]{Th2} K. Thomsen, {\em Phase transition in the CAR algebra}, Adv. Math. {\bf 372} (2020).







\end{thebibliography}
\end{document}